\documentclass[12pt]{article}

\usepackage[dvips]{graphicx}%
\usepackage{textcomp}
\usepackage{amsbsy}
\usepackage{latexsym}
\usepackage[mathscr]{eucal}
\usepackage{amsfonts}
\usepackage{amssymb}
\usepackage[usenames]{color}
\usepackage{amsmath,amsthm}

\usepackage{tikz}
\usetikzlibrary{arrows}
\usepackage{mathdots}

\usepackage{fullpage}

\newcommand{\id}{I}

\DeclareMathOperator{\rank}{rank}
\DeclareMathOperator{\range}{range}
\DeclareMathOperator{\linspan}{span}
\DeclareMathOperator{\essinf}{ess\,inf}

\providecommand{\bignorm}[1]{\bigl\lVert#1\bigr\rVert}
\providecommand{\Bignorm}[1]{\Bigl\lVert#1\Bigr\rVert}

\begin{document}
\bibliographystyle{plain}

\title
{Kolmogorov widths and low-rank approximations of parametric elliptic PDEs
\thanks{%
Research supported by the European Research Council under grant ERC AdG BREAD.
}
}
\author{ 
Markus Bachmayr and Albert Cohen
}
\hbadness=10000
\vbadness=10000
\vbadness=10000
\newtheorem{lemma}{Lemma}[section]
\newtheorem{prop}[lemma]{Proposition}
\newtheorem{cor}[lemma]{Corollary}
\newtheorem{theorem}[lemma]{Theorem}
\theoremstyle{definition}
\newtheorem{remark}[lemma]{Remark}
\newtheorem{example}[lemma]{Example}
\newtheorem{definition}[lemma]{Definition}
\newtheorem{proper}[lemma]{Properties}
\newtheorem{assumption}[lemma]{Assumption}
\newenvironment{disarray}{\everymath{\displaystyle\everymath{}}\array}{\endarray}

\def\RR{\rm \hbox{I\kern-.2em\hbox{R}}}
\def\NN{\rm \hbox{I\kern-.2em\hbox{N}}}
\def\ZZ{\rm {{\rm Z}\kern-.28em{\rm Z}}}
\def\CC{\rm \hbox{C\kern -.5em {\raise .32ex \hbox{$\scriptscriptstyle
|$}}\kern
-.22em{\raise .6ex \hbox{$\scriptscriptstyle |$}}\kern .4em}}
\def\vp{\varphi}
\def\<{\langle}
\def\>{\rangle}
\def\t{\tilde}
\def\i{\infty}
\def\e{\varepsilon}
\def\sm{\setminus}
\def\nl{\newline}
\def\o{\overline}
\def\wt{\widetilde}
\def\wh{\widehat}
\def\cT{{\cal T}}
\def\cA{{\cal A}}
\def\cI{{\cal I}}
\def\cV{{\cal V}}
\def\cB{{\cal B}}
\def\cF{{\cal F}}

\def\cR{{\cal R}}
\def\cD{{\cal D}}
\def\cP{{\cal P}}
\def\cJ{{\cal J}}
\def\cM{{\cal M}}
\def\cO{{\cal O}}
\def\Chi{\raise .3ex
\hbox{\large $\chi$}} \def\vp{\varphi}
\def\lsima{\hbox{\kern -.6em\raisebox{-1ex}{$~\stackrel{\textstyle<}{\sim}~$}}\kern -.4em}
\def\lsim{\hbox{\kern -.2em\raisebox{-1ex}{$~\stackrel{\textstyle<}{\sim}~$}}\kern -.2em}
\def\[{\Bigl [}
\def\]{\Bigr ]}
\def\({\Bigl (}
\def\){\Bigr )}
\def\[{\Bigl [}
\def\]{\Bigr ]}
\def\({\Bigl (}
\def\){\Bigr )}
\def\L{\pounds}
\def\pr{{\rm Prob}}
\newcommand{\cs}[1]{{\color{magenta}{#1}}}
\def\ds{\displaystyle}
\def\ev#1{\vec{#1}}     %
\newcommand{\lt}{\ell^{2}(\nabla)}
\def\Supp#1{{\rm supp\,}{#1}}
\def\R{\mathbb{R}}
\def\E{\mathbb{E}}
\def\nl{\newline}
\def\T{{\relax\ifmmode I\!\!\hspace{-1pt}T\else$I\!\!\hspace{-1pt}T$\fi}}
\def\N{\mathbb{N}}
\def\Z{\mathbb{Z}}
\def\P{\mathbb{P}}
\def\N{\mathbb{N}}
\def\Zd{\Z^d}
\def\Q{\mathbb{Q}}
\def\C{\mathbb{C}}
\def\Rd{\R^d}
\def\gsim{\mathrel{\raisebox{-4pt}{$\stackrel{\textstyle>}{\sim}$}}}
\def\sime{\raisebox{0ex}{$~\stackrel{\textstyle\sim}{=}~$}}
\def\lsim{\raisebox{-1ex}{$~\stackrel{\textstyle<}{\sim}~$}}
\def\divergence{\operatorname{div}}
\def\M{M}  \def\NN{N}                  %
\def\L{{\ell}}               %
\def\Le{{\ell^1}}            %
\def\Lz{{\ell^2}}
\def\Let{{\tilde\ell^1}}     %
\def\Lzt{{\tilde\ell^2}}
\def\Ltw{\ell^\tau^w(\nabla)}
\def\t#1{\tilde{#1}}
\def\la{\lambda}
\def\La{\Lambda}
\def\ga{\gamma}
\def\BV{{\rm BV}}
\def\Ga{\eta}
\def\al{\alpha}
\def\cZ{{\cal Z}}
\def\cA{{\cal A}}
\def\cU{{\cal U}}
\def\argmin{\mathop{\rm argmin}}
\def\argmax{\mathop{\rm argmax}}
\def\prob{\mathop{\rm prob}}
\def\A{\mathop{\rm Alg}}

\def \bphi{{\bf\phi}}

\def\cO{{\cal O}}
\def\cA{{\cal A}}
\def\cC{{\cal C}}
\def\cS{{\cal F}}
\def\bu{{\bf u}}
\def\bz{{\bf z}}
\def\bZ{{\bf Z}}
\def\bI{{\bf I}}
\def\cE{{\cal E}}
\def\cD{{\cal D}}
\def\cG{{\cal G}}
\def\cI{{\cal I}}
\def\cJ{{\cal J}}
\def\cM{{\cal M}}
\def\cN{{\cal N}}
\def\cT{{\cal T}}
\def\cU{{\cal U}}
\def\cV{{\cal V}}
\def\cW{{\cal W}}
\def\cL{{\cal L}}
\def\cB{{\cal B}}
\def\cG{{\cal G}}
\def\cK{{\cal K}}
\def\cS{{\cal S}}
\def\cP{{\cal P}}
\def\cQ{{\cal Q}}
\def\cR{{\cal R}}
\def\cU{{\cal U}}
\def\bL{{\bf L}}
\def\bl{{\bf l}}
\def\bK{{\bf K}}
\def\bC{{\bf C}}
\def\X{X\in\{L,R\}}
\def\ph{{\varphi}}
\def\D{{\Delta}}
\def\H{{\cal H}}
\def\bM{{\bf M}}
\def\bx{{\bf x}}
\def\bj{{\bf j}}
\def\bG{{\bf G}}
\def\bP{{\bf P}}
\def\bW{{\bf W}}
\def\bT{{\bf T}}
\def\bV{{\bf V}}
\def\bv{{\bf v}}
\def\bt{{\bf t}}
\def\bz{{\bf z}}
\def\bw{{\bf w}}
\def \meas {{\rm meas}}
\def\rhom{{\rho^m}}
\def\diff{\hbox{\tiny $\Delta$}}
\def\EE{{\rm Exp}}
\def\lll{\langle}
\def\argmin{\mathop{\rm argmin}}
\def\argmax{\mathop{\rm argmax}}
\def\dJ{\nabla}
\newcommand{\ba}{{\bf a}}
\newcommand{\bb}{{\bf b}}
\newcommand{\bc}{{\bf c}}
\newcommand{\bd}{{\bf d}}
\newcommand{\bs}{{\bf s}}
\newcommand{\bff}{{\bf f}}
\newcommand{\bp}{{\bf p}}
\newcommand{\bg}{{\bf g}}
\newcommand{\by}{{\bf y}}
\newcommand{\br}{{\bf r}}
\newcommand{\be}{\begin{equation}}
\newcommand{\ee}{\end{equation}}
\newcommand{\bex}{\begin{equation*}}
\newcommand{\eex}{\end{equation*}}
\newcommand{\bea}{$$ \begin{array}{lll}}
\newcommand{\eea}{\end{array} $$}
\def \Vol{\mathop{\rm  Vol}}
\def \mes{\mathop{\rm mes}}
\def \Prob{\mathop{\rm  Prob}}
\def \exp{\mathop{\rm    exp}}
\def \sign{\mathop{\rm   sign}}
\def \sp{\mathop{\rm   span}}
\def \vphi{{\varphi}}
\def \csp{\overline \mathop{\rm   span}}
\def \cost{\mathop{\rm   cost}}

\newcommand{\beqn}{\begin{equation}}
\newcommand{\eeqn}{\end{equation}}

\newenvironment{Proof}{\noindent{\bf Proof:}\quad}{\endproof}

\renewcommand{\theequation}{\thesection.\arabic{equation}}
\renewcommand{\thefigure}{\thesection.\arabic{figure}}

\makeatletter
\@addtoreset{equation}{section}
\makeatother

\newcommand\abs[1]{\left|#1\right|}
\newcommand\clos{\mathop{\rm clos}\nolimits}
\newcommand\trunc{\mathop{\rm trunc}\nolimits}
\renewcommand\d{d}
\newcommand\dd{d}
\newcommand\diag{\mathop{\rm diag}}
\newcommand\dist{\mathop{\rm dist}}
\newcommand\diam{\mathop{\rm diam}}
\newcommand\cond{\mathop{\rm cond}\nolimits}
\newcommand\eref[1]{{\rm (\ref{#1})}}
\newcommand{\iref}[1]{{\rm (\ref{#1})}}
\newcommand\Hnorm[1]{\norm{#1}_{H^s([0,1])}}
\def\int{\intop\limits}
\renewcommand\labelenumi{(\roman{enumi})}
\newcommand\lnorm[1]{\norm{#1}_{\ell^2(\Z)}}
\newcommand\Lnorm[1]{\norm{#1}_{L_2([0,1])}}
\newcommand\LR{{L_2(\R)}}
\newcommand\LRnorm[1]{\norm{#1}_\LR}
\newcommand\Matrix[2]{\hphantom{#1}_#2#1}
\newcommand\norm[1]{\left\|#1\right\|}
\newcommand\ogauss[1]{\left\lceil#1\right\rceil}
\newcommand{\QED}{\hfill
\raisebox{-2pt}{\rule{5.6pt}{8pt}\rule{4pt}{0pt}}%
  \smallskip\par}
\newcommand\Rscalar[1]{\scalar{#1}_\R}
\newcommand\scalar[1]{\left(#1\right)}
\newcommand\Scalar[1]{\scalar{#1}_{[0,1]}}
\newcommand\Span{\mathop{\rm span}}
\newcommand\supp{\mathop{\rm supp}}
\newcommand\ugauss[1]{\left\lfloor#1\right\rfloor}
\newcommand\with{\, : \,}
\newcommand\Null{{\bf 0}}
\newcommand\bA{{\bf A}}
\newcommand\bB{{\bf B}}
\newcommand\bR{{\bf R}}
\newcommand\bD{{\bf D}}
\newcommand\bE{{\bf E}}
\newcommand\bF{{\bf F}}
\newcommand\bH{{\bf H}}
\newcommand\bU{{\bf U}}
\newcommand\cH{{\cal H}}
\newcommand\sinc{{\rm sinc}}
\def\enorm#1{| \! | \! | #1 | \! | \! |}

\newcommand{\dm}{\frac{d-1}{d}}

\let\bm\bf
\newcommand{\bbeta}{{\mbox{\boldmath$\beta$}}}
\newcommand{\bal}{{\mbox{\boldmath$\alpha$}}}
\newcommand{\bbi}{{\bm i}}

\def\nnew{\color{Red}}
\def\mnew{\color{Blue}}

\maketitle
\date{}
\begin{abstract}
Kolmogorov $n$-widths and low-rank approximations are studied for families of 
elliptic diffusion PDEs parametrized by the diffusion coefficients. The decay of the 
$n$-widths can be controlled by that of the error achieved by best $n$-term approximations
using polynomials in the parametric variable. However, we prove that in certain relevant
instances where the diffusion coefficients are piecewise constant over a partition of 
the physical domain, the $n$-widths exhibit
significantly faster decay. This, in turn, yields
a theoretical justification of the fast convergence of reduced basis or POD methods 
when treating such parametric PDEs. Our results are confirmed by numerical
experiments, which also reveal the influence of the partition geometry on the
decay of the $n$-widths.
\end{abstract}

\section{Introduction}

Solving a parameter-dependent family of partial differential equations for a large number of different parameter values can be computationally demanding, since the solution of the problem for each single set of parameters typically already requires substantial resources. Put in abstract form,
one aims to find the solution $u$ of 
\be 
\label{paramdep}
\cP(u, y) = 0
\ee
for many different values of a vector $y=(y_1,\dots,y_d)$ in a certain range $U\subset \R^d$,
where $\cP$ is a partial differential operator $\cP$ parametrized by $y$. 

We assume that the problem is uniformly well posed in a separable Hilbert space $V$
whose elements depend on
a physical variable $x$ ranging in a domain $D\subset \R^m$: for each value of $y\in U$, the solution $u(y)$ 
belongs to $V$. 
We thus regard $u$ either as the function $y\mapsto u(y)$ from $U$ to $V$, or
as the function $(x,y)\mapsto u(x,y):=u(y)(x)$ from $D\times U$ to $\R$.

Many methods for reducing the complexity of such a task are explicitly or implicitly based on the existence of efficient approximations 
$u_n$ of $u$
of the general separable form
\be\label{generallowrank}
u_n(x,y):=  \sum_{k=1}^{n} v_k(x) \, \phi_k(y),
\ee
for some functions $\{v_1,\dots,v_n\}$ and $\{\phi_1,\dots,\phi_n\}$.
For instance, in the reduced basis method \cite{RHP}, one constructs $\{v_k\}_{k=1,\ldots,n}$ 
as particular instances $v_k=u(y^k)$ of the solution for well-chosen values $y^k\in U$, so that the problem
can be solved rapidly by applying, for each given $y$, the Galerkin method in the $n$-dimensional space spanned 
these basis functions. Similarly, the POD method \cite{KV} constructs $\{v_k\}_{k=1,\ldots,n}$ 
by applying a principal component analysis on a representative set of instances.
Both methods amount to an expansion of $u$ of the above form \eqref{generallowrank}, although the latter is not explicitly constructed
since the values $(\phi_k(y))_{k=1,\dots,n}$ are computed for each given $y$ by the Galerkin method.
In contrast, polynomial based methods such as in \cite{CDS1,CDS2} produce an explicit representation
 \eqref{generallowrank} where the functions $y\mapsto \phi_k(y)$ are multivariate polynomials. More generally,
 low-rank tensor methods \cite{KS} aim at directly building an approximation \eqref{generallowrank}
 with functions $v_k$ and $\phi_k$ constructed depending on $u$.

The efficiency of such approaches depends crucially on the size of $n = n(\varepsilon)$, which should not increase too rapidly 
with decreasing error tolerance $\varepsilon$ in the approximation \eqref{generallowrank}. A crucial question is how one can obtain bounds 
for $n(\varepsilon)$ in terms of $\varepsilon$ that are not crude overestimates, but indeed reflect the minimal required number 
of terms in an expansion of the form \eqref{generallowrank}. Such bounds can then give an indication to what extent the considered 
problem is amenable to such approximations.

First upper bounds can be obtained by placing restrictions on the choice of $\phi_k$.
A fairly well-studied instance are \emph{sparse polynomial expansions}. In this case, the functions $\phi_k$ are selected 
from an a priori chosen basis of tensor product polynomials, for instance 
\be\label{legendreexpansion}
    u_n(x,y) :=  \sum_{\nu \in \Lambda_n} u_\nu(x) \, L_\nu(y)   \,,
\ee
where each $L_\nu(y)=\prod_{i=1}^dL_{\nu_i}(y_i)$ is a tensor product of univariate Legendre polynomials of degrees $\nu_i$ in the variable $y_i$, and 
$\Lambda_n$ is a suitable subset of such multi-indices $\nu=(\nu_i)_{i=1,\dots,d}$ with $\#(\Lambda_n)=n$.
In this framework, one can also treat infinite-dimensional $U$, for instance $y \in [-1,1]^\N$, provided that the problem is anisotropic in the sense that the 
$y_i$ have decreasing influence as $i\to \infty$. We refer to \cite{CDS1,CDS2} where
approximations of the basic form \eqref{generallowrank} with provable convergence rates for the error
$\|u-u_n\|_{L^\infty(U,V)}$ as $n\to +\infty$
are derived in the infinite-dimensional framework. These approximations are derived by best $n$-term truncations of 
Legendre or Taylor expansions, see also \cite{BNTT1,BNTT2} for numerical methods in this line of ideas.

The question that we aim to address in this work is whether one can obtain provably more efficient approximations of the form \eqref{generallowrank} when the $\phi_k$ are \emph{not} selected from certain fixed basis functions, but are allowed to vary essentially arbitrarily.  We do not treat the case of infinite-dimensional $y$ here, but rather consider finite-dimensional parameter domains $U$ in an \emph{isotropic} setting, that is, with all parameters $y_i$ carrying equal weight.

When no limitations are imposed a priori on $\phi_k$, the relevant benchmark for reduced models that aim to guarantee uniform accuracy in $V$ for all values $y$ is the Kolmogorov $n$-width of the so-called {\it solution manifold} $u(U) \subset V$, which is defined for $n\in\N$ by
\be\label{nwidth}
 d_n(u(U))_V :=  \operatornamewithlimits{inf\vphantom{p}}_{\dim (W) = n} \; \sup_{v\in u(U)} \;  \operatornamewithlimits{min\vphantom{p}}_{w\in W} \norm{v - w}_V = \operatornamewithlimits{inf\vphantom{p}}_{\dim (W) = n} \; \sup_{y\in U} \;  \operatornamewithlimits{min\vphantom{p}}_{w\in W} \norm{u(y) - w}_V  \,.
\ee
In general, optimal subspaces $W_n$ corresponding to the infimum in \eqref{nwidth} are not easily determined, although greedy methods under certain assumptions provide reduced basis spaces for which the approximation errors have comparable rates of decay, see \cite{BCDDPW,DPW}.

A more accessible notion of low-rank approximation results from replacing $L^\infty(U, V)$ by $L^2(U,V)$, considering instead the quantities
\be\label{l2width}
\delta_n(u,\mu)_V^2:=\operatornamewithlimits{inf\vphantom{p}}_{\dim (W) = n} \; \int_U \;  \operatornamewithlimits{min\vphantom{p}}_{w\in W} \norm{u(y) - w}^2_V  \,d\mu(y)   \,,
\ee
where $\mu$ is a given probability measure on $U$. The subspaces for which the infimum 
is attained can in this case be characterized: $u\in L^2(U,V,\mu)$ induces a Hilbert-Schmidt operator 
from $L^2(U,\mu)$ to $V$ defined by
\be 
    \mathcal{M}_u: \varphi \mapsto \int_U u(y) \, \varphi(y)\, d\mu(y), 
\ee
with singular value decomposition
\be \label{Musvd}
  \mathcal{M}_u = \sum_{k=1}^\infty \sigma_k v_k \<  \cdot,\phi_k \>_{L^2(U,\mu)},  
\ee
where $\sigma:=(\sigma_k)_{k\geq 1} \in \ell^2(\N)$ is nonnegative and nonincreasing, 
and $(v_k)_{k\geq 1}$, $(\phi_k)_{k\geq 1}$ are orthonormal bases of $V$ and $ L^2(U,\mu)$, respectively. 
It is well known that for given $n$, the infimum in \eqref{l2width} is attained by choosing $W_n := \{  v_k \colon k = 1,\ldots,n\}$,
and has value
\be
\delta_n(u,\mu)_V^2 =\sum_{k>n} \sigma_k^2.
\label{deltasigma}
\ee
The operator $\mathcal{M}_u$ also yields a precise notion of rank for $u \in L^2(U, V)$ via
\be
  \rank(u) := \dim \range(\mathcal{M}_u) = \#\{ k\colon \sigma_k \neq 0 \} \,,
\ee
which, in general, is infinite. Note that $\mu$ may in principle be {\it any} probability measure. However, a
natural choice in the case where the $y_j$ are random variables is to take for $\mu$ the measure associated
to the distribution of $y$, since the quantity which is minimized in \iref{l2width} may then also be viewed as the
quadratic error $\E\bigl(\|u(y)-P_W u(y)\|_V^2 \bigr)$. When the $y_i$ are deterministic, a standard choice is to take 
for $\mu$ the uniform probability measure over $U$.

Note that since $\mu$ is a probability measure, we always have
\be
\delta_n(u,\mu)_V\leq d_n(u(U))_V.
\label{deltad}
\ee
On the other hand, polynomial approximation results such as in \cite{CDS1,CDS2}
yield estimates
for these quantities since obviously
\be
d_n(u(U))_V \leq \|u-u_n\|_{L^\infty(U,V)}\,,
\label{dnleg}
\ee
where $u_n$ is an $n$-term Legendre expansion of the form \iref{legendreexpansion}.

The results in this paper construct estimates for both \eqref{nwidth} and \eqref{l2width}
which are significantly sharper than those which could be obtained using the above estimate \iref{dnleg}.
The particular class of problems that we focus on are parametric diffusion equations
\be\label{strongform}
   -\divergence_x \bigl( a(x,y) \nabla u(x,y)  \bigr) = f(x)\,,  \quad x \in D \subset \R^m,\; y \in U := [-1,1]^d,
\ee
with $m,d\in \N$, where the coefficient $a$ has the form
\be \label{affinecoeff}
 a(x,y) = \bar{a}(x) - \sum_{i=1}^d y_i \psi_i(x)   \,.
\ee
Note that if the initial range of $y_i$ is a more general finite interval $I_i$, 
the range $y_i \in [-1,1]$ can always be ensured up to a renormalization of $\psi_i$ and a modification of $\bar a$.
While in most references there is a positive sign after $\bar a(x)$, we use here a negative sign
for later notational convenience. 

Throughout the paper, we assume $a$ to satisfy a \emph{uniform ellipticity assumption}: %
there exist $0< r \leq R$ such that
\begin{equation}  \label{uea}
   0 < r \leq a(x,y) \leq R < \infty  \,, \quad x\in D, \; y\in U .  
\end{equation}
In view of \eqref{affinecoeff}, this entails in particular
\be\label{uea2}  \sum_{i=1}^d \abs{\psi_i(x)} \leq \bar a(x) - r \,,\quad 
  \sum_{i=1}^d \abs{\psi_i(x)} \leq R - \bar a(x) \,,\quad x \in D\,. 
\ee
We consider the equation \eqref{strongform} in its weak form on $V:= H^1_0(D)$,
\be\label{weakform}
   \Bigl (\bar{A}  - \sum_{i=1}^d y_i A_i \Bigr) u(y)  = f \,,
\ee
where $f \in V'$ and $\bar{A}, A_i \colon V \to V'$ are defined by
\be
  \langle \bar A u, v\rangle :=  \int_D \bar a \nabla u \cdot \nabla v \,dx \,,\quad 
 \langle A_i u, v\rangle := \int_D \psi_i \nabla u \cdot \nabla v \,dx \,, \qquad u,v \in V \,.
\ee

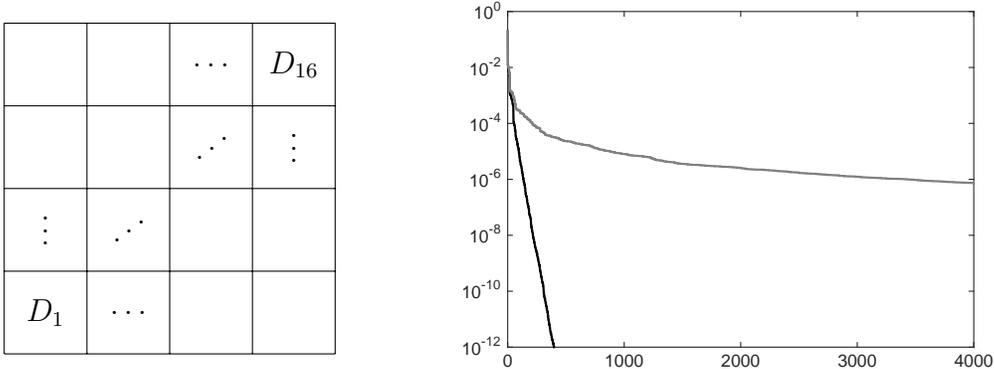
\begin{figure}
\centering
\raisebox{.5cm}{
\begin{tikzpicture}[scale=1.1]
\draw[step=1cm,black,thin] (0,0) grid (4,4);
\node at (0.5,0.5) {$D_1$};
\node at (3.5,3.5) {$D_{16}$};
\node at (.5,1.6) {$\vdots$};
\node at (1.5,.5) {$\hdots$};
\node at (1.5,1.6) {$\iddots$};
\node at (2.5,2.6) {$\iddots$};
\node at (2.5,3.5) {$\hdots$};
\node at (3.5,2.6) {$\vdots$};
\end{tikzpicture}} ~~~~~~~ \includegraphics[width=8cm]{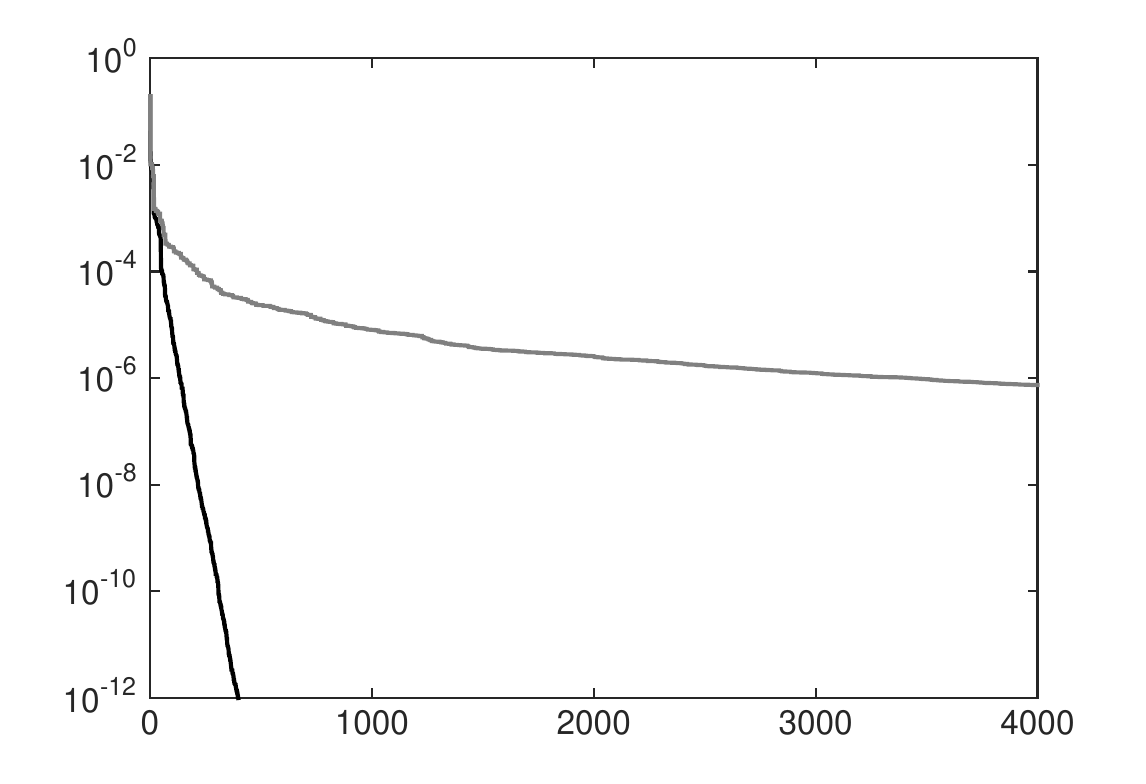}  
\caption{Example geometry of piecewise constant coefficients ($d=16$) and decay of singular values (black) and of ordered Legendre coefficients (gray),
see also Section 6.}
\label{fig:introdecay4x4}
\end{figure}
A simple, yet relevant, model problem that we study in detail -- and where it turns out that general low-rank approximations can indeed perform substantially better than sparse polynomial approximations -- is the case of piecewise constant coefficients $a$, for instance 
\be
\bar{a} = 1 \quad {\rm and }\quad \psi_i = \theta \Chi_{D_i},\;\theta<1,
\ee
where $\{D_1,\dots,D_d\}$ is a given partition of $D$. Such problems arise, for example, in the analysis of 
heat conductivity properties of components composed of several materials. 

As an example of what one may expect in such a case, Figure \ref{fig:introdecay4x4} shows a comparison between 
the numerically observed decay of (i) the singular values $\sigma_k$ as in \eqref{Musvd}
with $\mu$ being the uniform measure in $U$ and (ii) the reordered norms $\|u_\nu\|_V$ of Legendre coefficients in 
\iref{legendreexpansion} where the Legendre polynomials have been normalized in $L^2(U,\mu)$, 
in the case of a checkerboard partition of the unit square $D$ into $16$ smaller squares. 
Whereas the norms of the Legendre coefficients exhibit a subexponential decay, the singular values are observed to decay exponentially.
This indicates that, for this particular case, separable approximations of $u$ using optimally adapted 
$v_k$ and $\phi_k$, for example using reduced basis or POD methods, can perform substantially better than those obtained by best $n$-term truncation 
of Legendre series. This phenomenon was already observed in \cite{KT} and \cite{BG}. 
Here we provide a first theoretical explanation through a rigorous analysis for various situations.
We shall see in particular that this effect is significantly connected to the geometry of the partition.

The rest of our paper is organized as follows: in Section 2, we give two very simple examples for 
which the solution manifold $u(U)$ happens to be contained in a finite dimensional space,
meaning that $d_n(u(U))_V=0$ for $n$ sufficiently large, while polynomial approximations
require that the number of terms $n$ goes to $+\infty$ in order to converge to $u$. We then
consider in Section 3 the general elliptic problem \iref{strongform} with coefficients given by
\iref{affinecoeff}, and we consider its approximation by its truncated power series
\be
u_n(x,y)=\sum_{|\nu| \leq k} t_\nu(x)\, y^\nu,   \quad n:=n(d,k)={k+d\choose d}.
\ee
Our main observation is that under a general condition relating the functions $\bar a$ and $\psi_i$, 
this approximation has rank $r(n)$ significantly smaller than its number of terms $n$. Using the estimate
\be
d_{r(n)}(u(U))_V \leq \|u-u_n\|_{L^\infty(U,V)}\,,
\ee
this allows us to show that the $n$-widths of $u(U)$ decay significantly faster than the approximation error $\|u-u_n\|_{L^\infty(U,V)}$.
The general condition applies in particular to the case where $\bar a$ is constant and $\psi_i$ is proportional to $\Chi_{D_i}$, which
corresponds to the above mentioned case of piecewise constant coefficients. In Section 4, we specialize on the piecewise
constant situation and show that low-rank approximations can be obtained through the study of 
an auxiliary parametric problem posed on the skeleton of the partition and involving various Steklov-Poincar\'e operators.
Using this approach, we prove the exponential decay of the $n$-widths (which justifies the numerically observed
decay of the singular values) in the simple case of a $2\times 2$ checkerboard. We also demonstrate numerically
that this behaviour strongly depends on the geometry, in the sense that it is lost when considering a partition of $D$ into $4$
quadrilaterals which are not squares. Finally, in Section 5 we discuss the numerical schemes which we use in order to
perform the numerical experiments. These schemes combine non-adaptive polynomial approximations and 
rank reduction by singular value decompositions. The development of numerical schemes which combine rank reduction and
adaptive polynomial approximation, along the lines of \cite{BD}, is the subject of current investigation.

\section{Simple examples of exact low-rank representations}

We first consider two simple examples where explicit low-rank representations
can be given, leading to particularly favorable bounds for the Kolmogorov widths of solution manifolds,
when compared to error bounds for the polynomial approximations.

\begin{example}
A trivial instance of a situation where a sparse polynomial expansion in $y$ gives clearly suboptimal approximations is when all $\psi_i$ are proportional to $\bar{a}$, that is, when $\psi_i:= c_i \bar{a}$ with values $c_i\in \R$ such that \eqref{uea} holds. 
Then we may write $u(y)$ as a scalar multiple of $v_1:=u(0)=\bar{A}^{-1} f$ for all $y\in U$, that is
\be
u(x,y)  = v_1(x)\phi_1(y)\,,\quad \phi_1(y) := \Bigl( 1 - \sum_{i=1}^d c_i y_i  \Bigr)^{-1} \,,
\ee
which shows that $\rank(u) = 1$, or in other words $d_n(u(U))_V = 0$ for $n \geq 1$. However, $u$ has an infinite power series
\be
u(x,y)=\sum_{\nu\in\N^d_0} t_\nu (x) \,y^\nu, \quad t_\nu=\Biggl( \frac {|\nu| !}{\prod_{i=1}^d \nu_i !}\prod_{i=1}^d c_i^{\nu_i}\Biggr) v_1,
\ee
and therefore polynomial approximation error bounds are infinitely less favorable than the $n$-widths.
\end{example}

\begin{example}
We now consider a piecewise constant diffusion coefficient in the case of a one-dimensional interval $D$.
Without loss of generality we set $D=]0,1[$, assume $\{ D_1, \ldots, D_d\}$ to be a partition of $D$ into $d$ subintervals, and set $\bar{a}:= 1$ and $\psi_i := \theta\Chi_{D_i}$ for a $\theta \inÊ]0,1[$.  
In the interior of each $D_i$, we then have $ - (1 + \theta y_i) u'' = f$.
Hence, with any $F$ such that $F'' = f$, for each $y$ and $i$ we obtain
\be u(y)|_{D_i} \in \linspan\{ \Chi_{D_i} ,x\,\Chi_{D_i} , F\,\Chi_{D_i} \}\, . \ee
 This yields an upper bound of $3d$ degrees of freedom for $u(y)$. Since there are $d+1$ continuity conditions (independent of $y$) on $u$ at domain and subinterval boundaries, we have 
 \be
 \rank(u) \leq 2d -1,
\ee
 and therefore $d_{n} (u(U))_V = 0$ for $n \geq 2d-1$, independently of the particular choice of $f$.
 
The same argument still applies for general $\bar{a}$ that is such that $0<r\leq \bar a \leq R$ and $\psi_i:=\theta \bar{a} \Chi_{D_i}$,
for some $\theta \inÊ]0,1[$.  Since the $\psi_i$ have this particular form, on the interior of each $D_i$ we have the equation
\be
   - (1 + \theta y_i) (\bar{a} u' ) ' = f \,,
\ee
which can still be solved by integrating twice as before, and consequently one still has $\rank(u)\leq 2d-1$.
\end{example}

\begin{remark} 
One consequence of the previous example is an estimate of the $n$-widths of the solution 
manifold in a case where the parameter range is infinite dimensional. More precisely,
assume that the diffusion coefficient $a$ ranges in the set
\be
\cA= \bigl\{a \in C^{s}(\overline{D})\colon \|a\|_{C^s} \leq B, \; 0<r\leq a\leq R  \bigr\},
\ee
for some $B>0$, where $C^{s}(\overline{D})$ denotes the usual H\"older space for some $0<s\leq 1$. For any $a\in\cA$, 
there exists a sequence $(a_n)_{n\geq 1}$ of piecewise constant approximations for each uniform partition into $n$ subintervals such that 
\be
\norm{a - a_n}_{L^\infty} \leq B n^{-s}, \quad n\geq 1,
\ee 
and such that we also have $0<r\leq a_n\leq R$. Note that this rate is optimal since it is well known that
$d_n(\cA)_{L^\infty} \sim n^{-s}$. Denoting by $u(a)$ the corresponding solution to the diffusion equation
for this value of $a$, we know, see e.g. \cite{CD2}, that
\be
\norm{u(a) - u(a_n)}_{V} \leq r^{-2}\|f\|_{V'}\norm{a - a_n}_{L^\infty}\leq B r^{-2}\|f\|_{V'}\,n^{-s}.
\ee
We have seen that each $u(a_n)$ is a linear combination of $2n-1$ basis functions \emph{independent} of $a_n$.
It follows that
\be
d_n(u(\cA))_{V}\leq C n^{-s}, \quad n\geq 1.
\ee
In other words, in this particular case, the map $a \mapsto u(a)$ preserves the decay rate $n^{-s}$ between the $n$-widths
of $\cA$ in $L^\infty(D)$ and of $u(\cA)$ in $V=H^1_0(D)$, respectively. 
A result of this form was proved in \cite{CD} for a general holomorphic map
$a\mapsto u(a)$, however with a loss of $1$ in the value of $s$.
\end{remark}

\section{Using power expansions}\label{sec:powerexp}

We return to the general formulation \iref{weakform}.
It has been shown in \cite{CDS2} that under the assumption \eqref{uea}, $u$ has a Taylor series expansion 
\be
  u(y) = \sum_{\nu  \in  \N_0^d}  t_\nu y^\nu \,, \quad t_\nu = \frac{1}{\nu!} \partial_\nu u(0) \,,
\ee
converging unconditionally in $L^\infty(U,V)$. The Taylor coefficients $t_\nu$ can be computed by differentiating
the equation
\be
\(\bar A+\sum_{i=1}^d y_iA_i\) u(y)=f,
\ee
at $y=0$. For $\nu=0$ the null multi-index, we have $t_0 = \bar{A}^{-1} f$ and for other values of $\nu$,
application of the Leibniz rule gives
\be
t_\nu = \sum_{i\colon \nu_i \neq 0} \bar{A}^{-1} A_i\,  t_{\nu - e_i} \,,
\ee
where $e_i$ denotes the Kroenecker vector with $1$ at position $i$. Introducing the abbreviations 
\be 
  B_i := \bar{A}^{-1} A_i   \,, \quad  g:= \bar{A}^{-1}f  \,,
\ee
we find that the partial Taylor sums can be re-expressed in the form
\be\label{neumannpartial}
   u_k(y) := \sum_{\abs{\nu}\leq k} t_\nu y^\nu = \sum_{\ell = 0}^k \(  \sum_{i=1}^d y_i B_i \)^\ell g \,.  
\ee
Thus the Taylor expansion can also be interpreted as a Neumann series,
\be\label{neumannseries}
  u(y) =   \Bigl( I- \sum_{i=1}^d y_i B_i \Bigr)^{-1} \bar{A}^{-1} f  = \sum_{\ell = 0}^\infty \Bigl(  \sum_{i=1}^d y_i B_i \Bigr)^\ell g  \,.
\ee

We now introduce the inner product $\langle \cdot, \cdot \rangle_{\bar{A}} := 
\langle \bar{A}\,\cdot, \cdot\rangle$, and its corresponding norm $\|\cdot\|_{\bar A}$, on $V = H^1_0(D)$, 
which in the case $\bar{a} = 1$ 
coincide with the standard inner product and norm of $V$. More generally, we have the norm equivalence
\be
r \|v\|_V^2 \leq \|v\|_{\bar A}^2\leq R\|v\|_V^2.
\ee
The following proposition yields an estimate for the rate of convergence of the series \eqref{neumannseries}.

\begin{prop}\label{propcontraction}
Under the assumption \eqref{uea}, we have
\be\label{contraction}
    \Bignorm{ \sum_{i=1}^d y_i B_i v }_{\bar{A}}  \leq   \rho  \norm{  v  }_{\bar{A}} \,,\quad v\in V,\;  y \in U ,
\ee
where
\be
\rho:= 1 - \frac{r}{\norm{\bar{a}}_{L^\infty}} <1.
\ee
\end{prop}

\begin{proof}
Let $u:= \sum_{i=1}^d y_i B_iv$, then by definition
\be
    \int_D  \bar{a} \nabla u \cdot \nabla w \, dx = \sum_{i = 1}^d y_i \int_D \psi_i \nabla v \cdot \nabla w \, dx \,,\quad w\in V .
\ee
Taking $w = u$, and using the Cauchy-Schwarz inequality, we obtain
\be
  \|u\|_{\bar{A}}^2 \leq   \|u\|_{\bar{A}}\, \|v\|_{\bar{A}} \,   \sup_{x\in D} \frac1{\bar{a}(x)} \sum_{i=1}^d \abs{\psi_i(x)}.
\ee
Using \eqref{uea2} we arrive at the assertion.
\end{proof}

Note that the total number of separable terms in $u_k$ is bounded by  
\be
n=n(d,k) :=  \#\{ \nu \colon \abs{\nu} \leq k\}  =  {{k+d}\choose{d}} \,,
\ee
and we have $\frac{1}{d!} k^d \leq n(d,k)\leq (k+1)^d$. Hence in view of Proposition \ref{propcontraction}, $u_k$ provides an approximation
with error 
\be
  \norm{ u - u_k}_{L^\infty(U,V)} \leq \frac{ \norm{g}_{\bar{A}} }{1-\rho}  \, \rho^{k+1} \,,\quad   \rho := 1 - \frac{r}{\norm{\bar{a}}_\infty},
  \label{approuk}
\ee
and with rank
\be
\rank(u_k)\leq n(d,k) \leq (k+1)^d.
\label{genericbound}
\ee
Setting $k(d,n) := \max\{ k \in \N \colon n(d,k) \leq n \} \geq n^{\frac1d} - 1$, 
we find on the one hand that
\be
\rank(u_{k(d,n)})\leq n,
\ee
and on the other hand that
\be   \bignorm{ u - u_{k(d,n)}}_{L^\infty(U,V)}   \lesssim  \rho^{ k(d,n)+1} \leq e^{- \abs{\ln \rho} \,n^{1/d} }.  \ee
In summary, \eqref{neumannpartial} thus provides the bound
\be  \label{trivialnwidth}
    d_n\bigl(u(U)\bigr)_V  \lesssim  e^{- \abs{\ln \rho} \,n^{1/d} }  
\ee
for the $n$-widths of the solution manifold $u(U)$, where the multiplicative constant depends on $g$ and $\rho$.

As an example of a rather general structural property of the problem that leads to a reduction in the $n$-widths, we consider the additional assumption
\be \label{sumAi}
  \sum_{i=1}^d A_i = \theta \bar{A} \,.
\ee
In our setting, this holds if $\psi_i = \bar{a} \varphi_i$ with the $\varphi_i$ forming a scaled partition of unity, that is, $\sum_{i=1}^d \varphi_i = \theta$ for a $\theta \in ]0,1[$.
The case of piecewise constant $a$ is a particular instance, where $\bar{a}$ is constant and $\psi_i =\bar a  \theta \Chi_{D_i}$ with $\{ D_1,\ldots, D_d\}$ a partition of $D$.

Note that \eqref{sumAi} is equivalent to $\sum_{i=1}^d B_i = \theta\, \id$.
Let us consider first the consequences in the case $d=2$, where this implies
\be  u_k(y) = \sum_{j=0}^k (y_1 B_1 + y_2 B_2)^j g = \sum_{j=0}^k \bigl(\theta y_2  \id + (y_1-y_2) B_1\bigr)^j g  \,.
\ee
We can rewrite each term in the sum as
\be
(y_1 B_1+y_2 B_2)^j  g= (\theta y_2 \,\id +(y_1-y_2) B_1)^j g
= \sum_{\ell=0}^j (\theta y_2)^{j - \ell} {j \choose \ell}(y_1-y_2)^\ell B_1^\ell g.
\ee
Exchanging the order of summations in the sums over $\ell$ and $j$, setting
\be
v_\ell := B_1^\ell g\,,\quad  \phi_{k,\ell} (y):=(y_1-y_2)^\ell \sum_{j=\ell}^k (\theta y_2)^{j - \ell}  {j \choose \ell}
\ee
we thus obtain
\be
 u_k(y)  = \sum_{\ell =0}^k \phi_{k,\ell} (y) \, v_\ell   \,.
\ee
Consequently, under the assumption \eqref{sumAi} the rank of $u_k$ has the linear bound
\be
\rank(u_k)\leq k+1,
\ee
which is obviously more favorable than the quadratic bound given by \iref{genericbound} in the generic case.
The same reasoning can be applied when $d > 2$, leading to the following result.

\begin{prop}\label{prp:generalreduction}
If \eqref{sumAi} holds, for each $k$ there exist $v_\ell \in V$ and $d$-variate polynomials $\phi_\ell$, with $\ell =1,\ldots, n(d-1, k)$, such that
\be
  u_k(y) = \sum_{\ell=1}^{n(d-1, k)} \phi_{k,\ell} (y) \,v_\ell,
\ee
and therefore 
\be
\rank(u_k)\leq n(d-1, k) \leq (k+1)^{d-1}.
\ee
\end{prop}

\begin{proof}
Analogously to the above considerations for $d=2$, we obtain
\be\label{uKn}
  \sum_{j=0}^k \Bigl(  \sum_{i=1}^d y_i B_i  \Bigr)^j   =  \sum_{\ell=0}^k  \gamma_{k,\ell}(y_d)  \Bigl(\sum_{i =1}^{d-1} (y_i - y_d) B_i  \Bigr)^\ell\,,\quad \gamma_{k,\ell}(y_d) :=   \sum_{j = \ell}^k {j \choose \ell} (\theta y_d)^{j - \ell}  \,.
\ee
For notational convenience, we define the operators $\tilde B^{(\ell)}_{\kappa}$ for multi-indices $\kappa = (\kappa_1,\ldots, \kappa_{d-1})$ 
such that
\be
   (y_1 B_1 + \ldots + y_{d-1} B_{d-1})^\ell = \sum_{\abs{\kappa} = \ell}  y^\kappa \tilde B^{(\ell)}_{\kappa}  \,, \quad y^\kappa := y_1^{\kappa_1} \cdots y_{d-1}^{\kappa_{d-1}}\,.
\ee
Note that the $\tilde B^{(\ell)}_{\kappa}$ are thus precisely the sums of powers of $B_i$, $i=1,\ldots,d-1$, that arise in a Taylor expansion of $u(y)$ with respect to $y_1,\ldots, y_{d-1}$.
Using this to rewrite the right hand side in \eqref{uKn}, we arrive at
\be 
  u_k(y) = \sum_{\abs{\kappa} \leq k}    \gamma_{k, \abs{\kappa}}(y_d) \, y^\kappa  \, \tilde B^{(\abs{\kappa})}_\kappa g \,.
\ee
Let $\{ \kappa_\ell \}_{\ell=1,\ldots,\hat n(d-1, k)}$ be an enumeration of $\{ \abs{\kappa} \leq k \}$. We then obtain the assertion by choosing
\bex
   v_\ell :=  \tilde B^{(\abs{\kappa_\ell})}_{\kappa_\ell} g \,,\quad  \phi_{k,\ell} (y) :=     \gamma_{k, \abs{\kappa_\ell}}(y_d) \, y^{\kappa_\ell} \,. \qedhere
\eex
\end{proof}

As can be seen from the proof, the functions $v_k$ in Proposition \ref{prp:generalreduction} are precisely the Taylor coefficients $t_{(\kappa_1,\ldots,\kappa_{d-1},0)}$ of $u(y)$. As a consequence of the proposition, under the additional assumption \eqref{sumAi}, the generic estimate \eqref{trivialnwidth} improves to
\be  
    d_n\bigl(u(U)\bigr)_V  \lesssim  e^{- \abs{\ln \rho} \, n^{1/(d-1)} }    \,.
    \label{improvedtrivialnwidth}
\ee
Note that in view of \iref{deltasigma} and \iref{deltad}, this estimate implies that the singular values in \iref{Musvd} satisfy
\be
\sum_{k>n} \sigma_k^2 \lesssim e^{- 2\abs{\ln \rho} \, n^{1/(d-1)} }
\ee
and in particular
\be
\sigma_n  \lesssim  e^{- \abs{\ln \rho} \, n^{1/(d-1)} }.
\label{subexpsing}
\ee
In these estimates, the multiplicative constants depend on $g$ and $\rho$, similar to \eqref{trivialnwidth}.

\section{Using the geometry of the partition}\label{sec:geom}

As noted above, the assumption \eqref{sumAi} can hold under more general conditions than a piecewise constant coefficient $a$. The reduction in the
bound for the $n$-widths between \iref{trivialnwidth} and  \iref{improvedtrivialnwidth}, however, weakens rapidly as $d$ increases. In particular, the number of terms required to meet a certain error bound still grows exponentially in $d$. In order to obtain stronger, but also more specialized results, 
we now focus on certain parametric forms of $a$ which are described through a partition of $D$.

\subsection{Steklov-Poincar\'e operators}

For given $\bar{a} \in L^\infty(D)$ with $\essinf \bar{a} >0$, we consider $\psi_j := \theta \Chi_{D_j} \bar{a}$ with a fixed $\theta \in ]0,1[$, where for the moment, $\{D_i\} $ is a partition of $D$ into general Lipschitz subdomains. 
We denote the skeleton of this partition by
\be
  \Gamma := \bigcup_{i=1}^d \partial D_i \setminus \partial D  
\ee
and define $V_\Gamma$ as the space of trace values of functions in $V$ on $\Gamma$, also denoted by 
$V_\Gamma:= H^{\frac12}_{00}(\Gamma)$, which we equip with the norm
\be
\|w\|_{V_\Gamma}:= \min\bigl\{ \|v\|_{\bar A}\colon v\in V, \; v_{|\Gamma}=w \bigr\}.
\ee
We also introduce for each $i=1,\dots,d$ the space $V_i := H^1_0(D_i)$ viewed as a closed subspace
of $V=H^1_0(D)$ by regarding its elements as extended by zero to all of $D$. The direct sum $V_1\oplus\cdots \oplus V_d$
is a closed subspace of $V$, and we denote by $W$ its $\bar{A}$-orthogonal complement in $V$.

Furthermore, we introduce for each $i=1,\dots,d$ the $\bar A$-harmonic extension operator $E_i$ from $V_\Gamma$ to the space
\be
H^1(D_i)\cap H^1_0(D):=\{ v\in H^1(D_i) \colon v|_{\partial D} = 0\},
\ee
which maps the trace value $v_\Gamma \in V_\Gamma$ to the extension $E_i v_\Gamma$  defined by 
\be
  \int_{D_i} \bar{a} \nabla E_i v_\Gamma \cdot \nabla w\,dx = 0 \quad\text{for all $w\in V_i$}, \qquad E_i v|_{\partial D_i \cap \partial D} = 0,\;  E_i v|_{\partial D_i\cap\Gamma} = v_\Gamma |_{\partial D_i} \,.
\ee
In the case where $\bar a$ is constant, this is the usual harmonic extension. We define the operator $E$ from $V_\Gamma$ to $V$ that 
concatenates these extensions into a function defined on $D$ by
\be  E\colon  v_\Gamma \mapsto E v_\Gamma \,,\qquad E v_\Gamma|_{D_i} := E_i v_\Gamma ,\; i=1,\ldots, d\,. \ee
Note that the space $W$ coincides with the range of $E$.

The operators $B_i=\bar{A}^{-1}A_i$ are self-adjoint in the $\bar{A}$-inner product, and for each $i$ they have the properties
\be \label{Binvariance}
  B_i|_{V_i} = \id, \qquad B_i|_{V_j} = 0 \;\text{ for $j\neq i$.}
\ee
Consequently, these operators also map $W$ to itself. 

We now decompose $u(y)$ into the sum of its $\bar A$-orthogonal projections onto the $V_i$ and $W$, according to
\be\label{orthdecomp}
  u (y)=  u_W(y) + \sum_{i=1}^d u_i(y) \,.
\ee
The component $u_i(y) \in V_i$ solves
\be \label{uivarform}
   (1-\theta y_i)\int_{D_i} \bar{a} \nabla u_i(y) \cdot\nabla v \,dx = \int_{D_i} f \,v\,dx \,,\; v \in V_i, \quad\text{and}\quad u_i(y) |_{D\setminus D_i} = 0 \,. 
\ee
Each $u_i$ is therefore of rank at most $1$, depending only on $y_i$, and $\sum_{i=1}^d u_i$ therefore has rank at most $d$.

The component $u_W$ can also be written in the form $u_W(y)=E u_\Gamma(y)$ for a $u_\Gamma(y) \in V_\Gamma$, 
where $u_\Gamma(y)$ is characterized by the variational problem
\begin{equation}\label{eq:complement}
 \sum_{i=1}^d (1 - \theta y_i) \int_{D_i} \bar{a} \nabla E u_\Gamma(y) \cdot \nabla E v_\Gamma \,dx 
   = \int_D f \,E v_\Gamma \, dx \,,\quad v_\Gamma \in V_\Gamma  \,.
\end{equation}
To obtain a more concise formulation of the equation for $u_\Gamma$, we introduce the Steklov-Poincar\'e operators $S_i\colon V_\Gamma \to V'_\Gamma$, $i= 1,\ldots,d$, and $\bar{S}\colon V_\Gamma \to V_\Gamma'$ defined by
\be   \langle S_i v_\Gamma, w_\Gamma\rangle := \int_{D_i} \bar{a} \nabla E v_\Gamma\cdot \nabla E w_\Gamma\,dx, \quad v_\Gamma, w_\Gamma\in V_\Gamma,
\ee
and 
\be 
\bar{S} := \sum_{i=1}^d S_i \,.
\ee
The norm of $V_\Gamma$ is induced by the inner product
  \be ( u, v)_\Gamma := \langle \bar{S} u,v\rangle  = \int_D \bar{a} \nabla Eu \cdot \nabla E v\,dx\,. \ee
In particular, the operator $\bar{S}$ is bounded, elliptic, and defines an isometry between $V_\Gamma$ and $V'_\Gamma$.

Using the operators $S_i$ and $\bar{S}$, we can now rewrite \eqref{eq:complement} as
 \be  \Bigl(\bar{S} - \theta \sum_{i=1}^d y_i S_i\Bigr) u_\Gamma(y) =f_\Gamma  \,,  \ee
where $f_\Gamma \in V_\Gamma'$ is defined by 
\be     
     \langle f_\Gamma , v_\Gamma \rangle := \int_D f \, E v_\Gamma \,dx  \,,\quad v_\Gamma \in V_\Gamma \,. 
\ee
We thus obtain the Neumann series representation
\be\label{neumannpartialif}
  u_{\Gamma} = \lim_{k\to \infty} u_{k,\Gamma}  , \quad \quad  u_{k,\Gamma}(y) := \sum_{\ell = 0}^k  \Bigl( \sum_{i=1}^d y_i \,(\theta \bar{S}^{-1} S_i)  \Bigr)^\ell g_\Gamma, 
   \ee
   where $g_\Gamma:=\bar S^{-1}f_\Gamma$. By a similar argument as in the proof of Proposition \ref{propcontraction}, we obtain that
   \be\label{contractiongamma}
    \Bignorm{ \sum_{i=1}^d y_i  \theta \bar{S}^{-1}S_i v }_{V_\Gamma}  \leq    \theta \norm{  v  }_{V_\Gamma} \,,\quad v\in V_\Gamma, \;y \in U,
\ee
which shows that the Neumann series converges in $L^\infty(U, V_\Gamma)$.

Note that the Neumann series for $u$ can be rewritten as
\be
u_k(y)=Eu_{k,\Gamma}(y) +\sum_{i=1}^d u_{k,i}(y),
\ee
where $u_{k,i}(y)$ is the $\bar A$-orthogonal projection of $u_k(y)$ onto $V_i$. Similar to $u_i$, the rank of $u_{k,i}$ is at most $1$. Therefore
\be\label{rankif}
\rank(u_k)  \leq d + \rank(u_{k,\Gamma}) \,.
\ee
The task of obtaining favorable bounds for the rank of $u_k$
is thus in the present setting essentially reduced to obtaining similar such bounds for
the rank of $u_{k,\Gamma}$, which is defined on $\Gamma\times U$. 
We shall see next that what one obtains in this regard can depend 
strongly on the particular geometry of $D$ and $\Gamma$.

\subsection{A particular example: four squares}\label{sec:2x2}

We now turn to the case of piecewise constant coefficients, with each piece varying within the same range of values, where we may assume without loss of generality that $\bar{a} = 1$. 
We consider the case of the unit square $D = ]-\frac12, \frac12[^2$, partitioned into the four symmetric quadrants,
\be \label{part2x2}
 \textstyle   D_1 = ]-\frac12, 0[^2,\quad D_2 = ]0, \frac12[\times ]-\frac12,0[,\quad D_3 =]-\frac12 , 0[\times ]0, \frac12[,\quad D_4 = ]0, \frac12 [^2  \,,  
\ee
and with $\psi_i := \theta \Chi_{D_i}$, $\theta \in ]0,1[$.

\begin{figure}\centering
\begin{tabular}{cc}
\includegraphics[width=7cm]{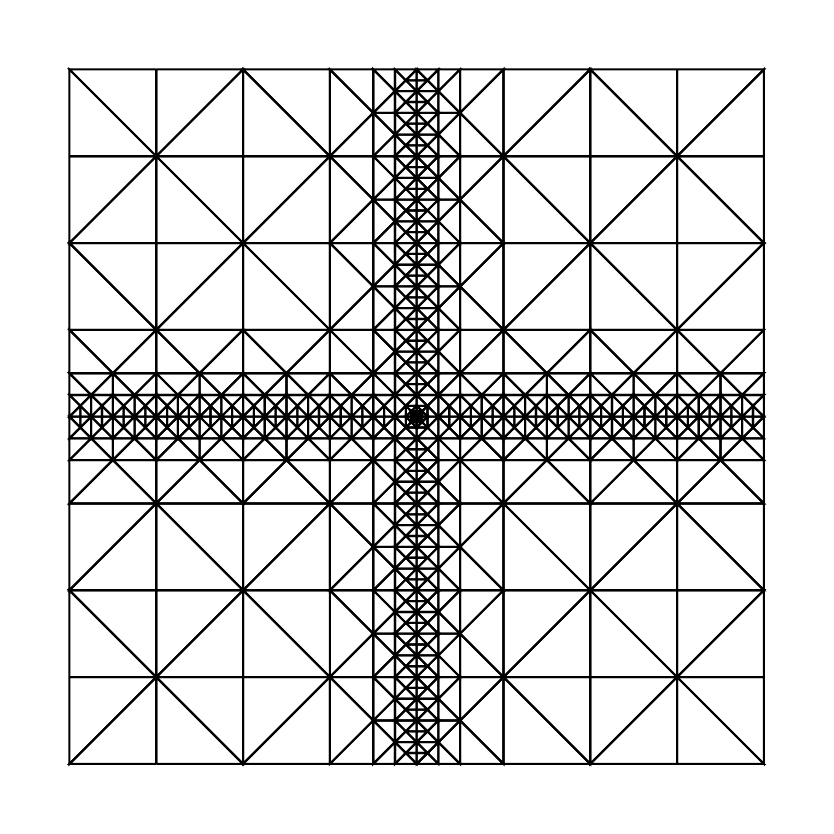} & \includegraphics[width=7cm]{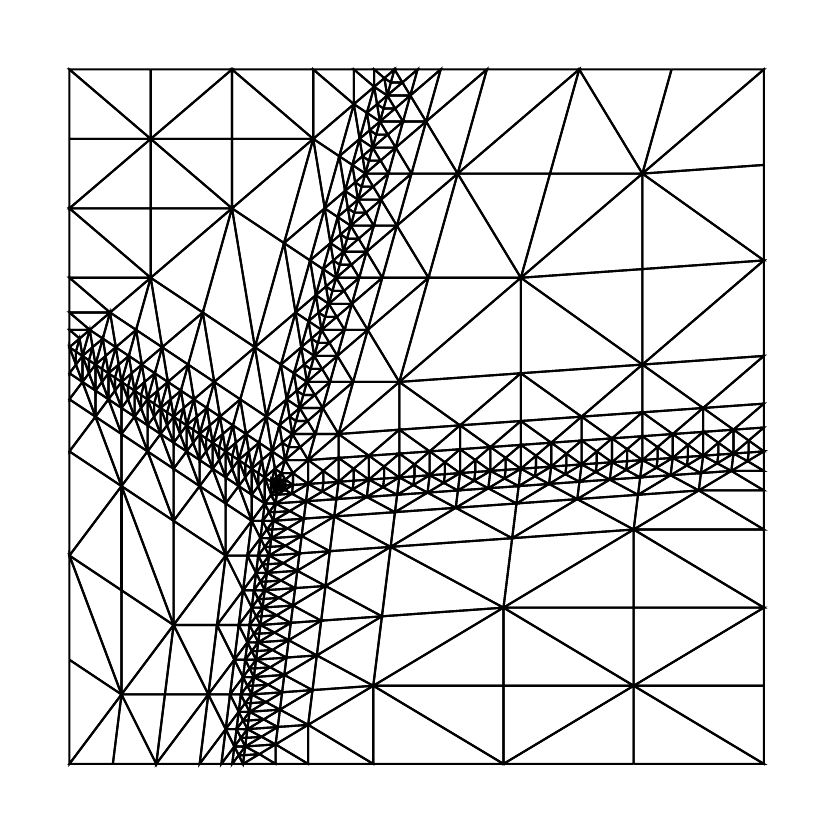} \\
 (a)  &  (b)
\end{tabular} 
\caption{Example meshes for the two test geometries with 992 triangles and 521 vertices. The meshes used in the actual computation have 119\,272 triangles and 59\,909 vertices with a similar gradation towards subdomain boundaries.}\label{fig:meshes}
\end{figure}

In order to motivate what follows, we first consider a numerical experiment, where we compare the decay
 of the $V$-norms of the tensor product Legendre coefficients $u_\nu$ as in \eqref{legendreexpansion} 
 to the decay of singular values $\sigma_k$ in the decomposition \eqref{Musvd}. We execute these computations
 under a {\it fixed} space discretization through a finite element space $V_h$: the exact solution map $y\mapsto u(y)$
 is replaced by $y\mapsto u_h(y)$, where $u_h(y)$ is the Galerkin projection of $u(y)$ onto the space $V_h$. 
 We compare the results obtained for (a) the symmetric partition \eqref{part2x2} to those for (b) a non-symmetric
partition into four quadrilaterals which may be viewed as a distorted version of \eqref{part2x2}. As explained further in Section 5, 
 in order to obtain result that reflect the behavior of the Legendre coefficients and singular values for the
 original solution map $y\mapsto u(y)$, it is extremely important to resolve the subdomain interfaces 
 appropriately in the spatial discretization. Figure \ref{fig:meshes} displays the typical meshes that
 we use for the symmetric partition \eqref{part2x2} and the distorted partition.
 We employ piecewise linear finite elements on further refined meshes, which gives 59\,365 degrees of freedom, 4\,105 of these on subdomain interfaces.
For the discretization in the parametric variable, we use Galerkin projection onto tensor product Legendre polynomials of total degree
at most 15. We therefore compute ${19 \choose 4}=3876$ Legendre coefficients $u_{h,\nu}$, each of which is a finite 
element function from $V_h$. A detailed discussion of the precise computational procedure is deferred to Section \ref{sec:numerics}.
The given results are obtained with the choices $f := 1$ and $\theta := \frac12$.

In Figure \ref{fig:decay2x2}, we compare the decay of the $V$-norms of the Legendre coefficients $u_{h,\nu}$ 
to the decay of singular values $\sigma_k$ in the decomposition \eqref{Musvd} 
where we have replaced $V$ by $V_h$ and $u(y)$ by $\sum_{|\nu|\leq 15} u_{h,\nu} L_\nu(y)$ in the definition of $\cM_u$.
In both cases (a) and (b), the second sequence decays substantially faster. However, although the Legendre coefficients behave almost identically in the two cases, we observe an exponential decrease of the $\sigma_k$ in case (a), but only a subexponential decrease in case (b).
A closer inspection shows that the decay in case (b) is in fact not substantially better than guaranteed by \iref{subexpsing}.

As explained in Section 5, the Galerkin discretization in the Legendre basis up to some total order $\ell$ allows us to
exactly compute the truncated power series $u_{h,\ell}(y)=\sum_{|\nu|\leq \ell} t_{h,\nu}y^\nu$ associated to the map $y\mapsto u_h(y)$.
Figure \ref{fig:itersvs2x2} displays the ranks of the partial sums $u_{h,k}$ for $k=1,\ldots, 10$ 
together with the corresponding singular values.  In case (a), these ranks are observed to increase approximately linearly in $k$ before the smallest non-zero singular values drop below machine precision, whereas the ranks grow more rapidly (faster than quadratically) in case (b).
Further experiments reveal that there is no qualitative change in the results if one prescribes a less regular right hand side $f$,
which shows that the decay of singular values is not tied to the spatial smoothness of $u$.

\begin{figure}
\begin{tabular}{cc}
\hspace{-9pt}\includegraphics[width=8cm]{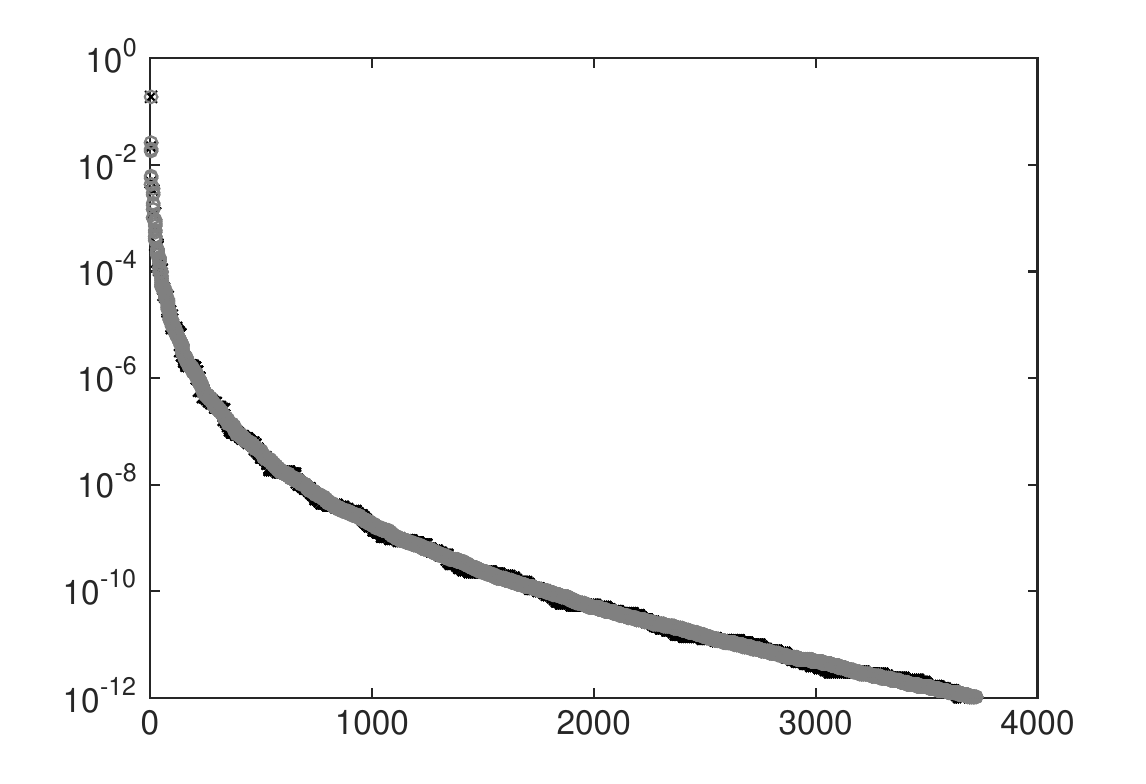} & \includegraphics[width=8cm]{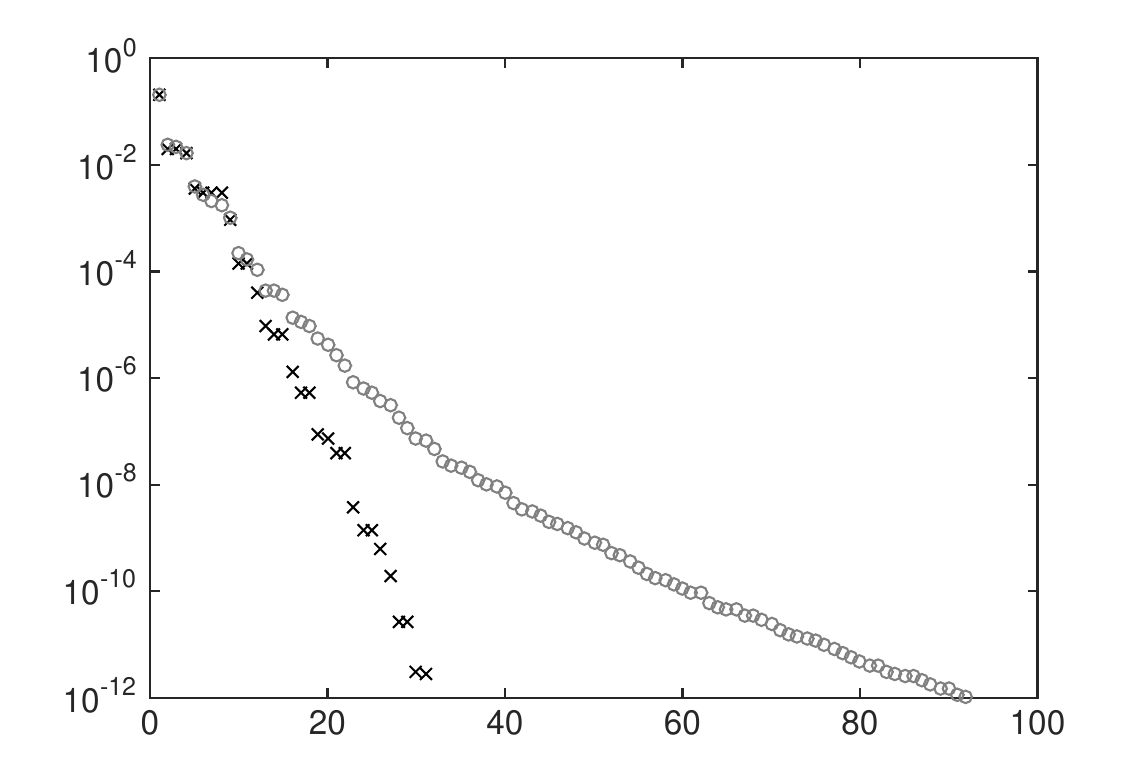} \\
\footnotesize Sorted $H^1$-norms of Legendre coefficients  &  \footnotesize Singular values
\end{tabular} 
\caption{Decay of sorted $H^1$-norms of Legendre coefficients and of singular values, both for geometries (a) and (b) in Figure \ref{fig:meshes} (marked by $\times$ and $\textcolor{gray}{\circ}$, respectively).}
\label{fig:decay2x2}
\end{figure}

\begin{figure}
\begin{tabular}{cc}
\hspace{-9pt}\includegraphics[width=8cm]{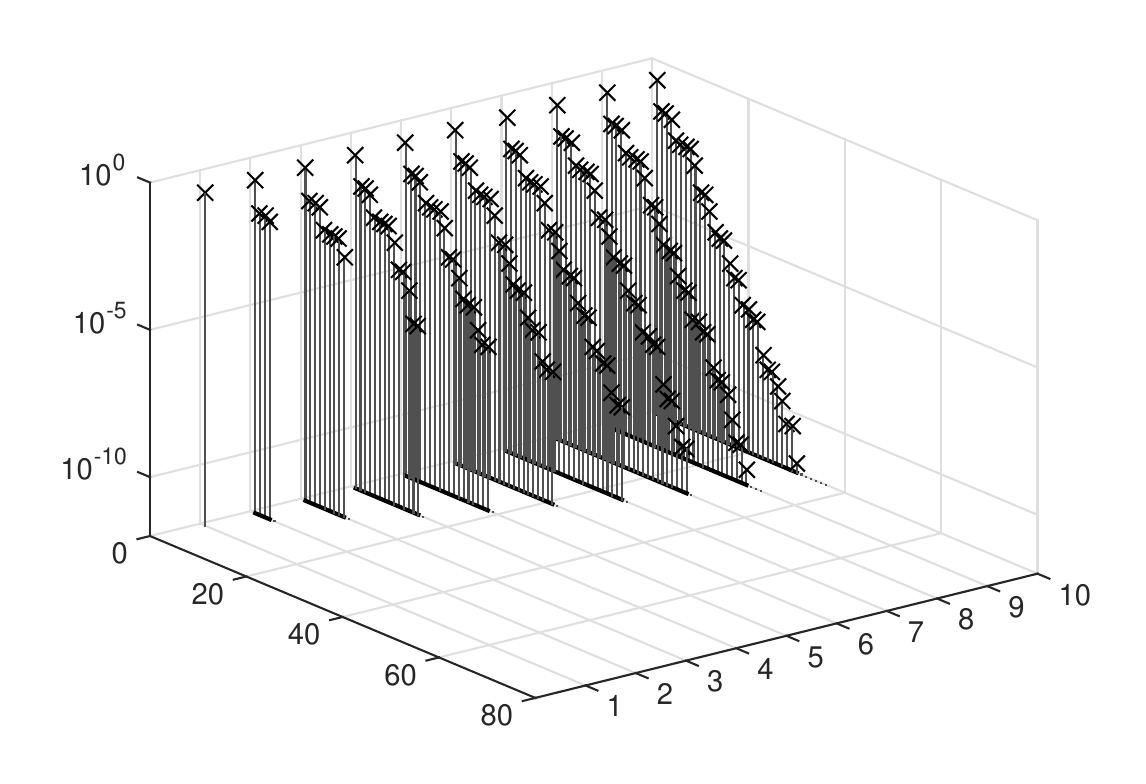} & \includegraphics[width=8cm]{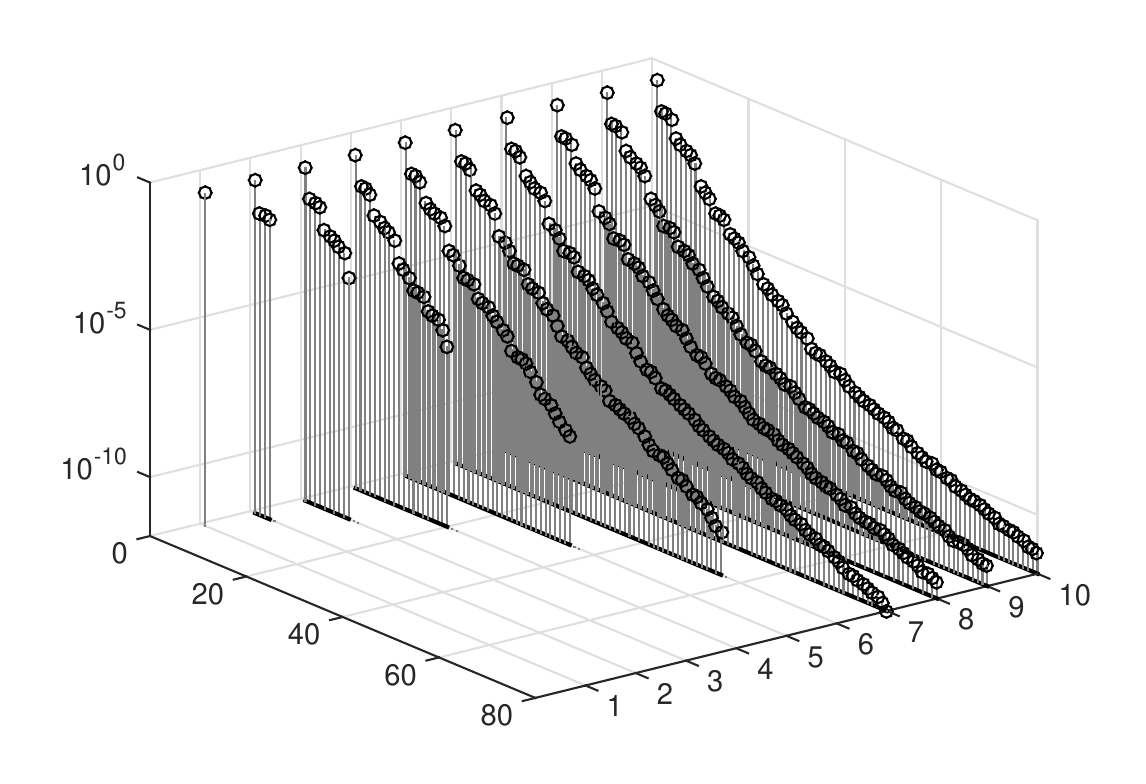} \\
\footnotesize (a)  &  \footnotesize (b)
\end{tabular} 
\caption{Ranks and first 80 singular values of the partial sums of the Neumann series for orders $k=1,\ldots,10$, and for geometries (a) and (b) in Figure \ref{fig:meshes} (marked by $\times$ and $\textcolor{gray}{\circ}$, respectively).}\label{fig:itersvs2x2}
\end{figure}

Our aim is now to explain the substantially better low-rank approximability observed in the case of a symmetric geometry.
The remainder of this section is devoted to the proof of the following theorem, which requires several preparatory steps.
Note that under our present assumptions, $\norm{\cdot}_{\bar A} $ equals the standard norm on $H^1_0(D)$.

\begin{theorem}\label{thm:rankest2x2}
Let $d=4$, and let $D_i$, $i=1,\ldots,4$, be defined as in \eqref{part2x2}. For each $k\in \N$, and for $n(k):= 8k+5$, there exist $v_1,\ldots,v_{n(k)} \in V$ and $d$-variate polynomials $\phi_1,\ldots,\phi_{n(k)}$, each of
total degree $k$, such that the truncated power series $u_k(y)=\sum_{\abs{\nu} \leq k} t_\nu y^\nu$ has the exact expression
\be
u_k(y)=\sum_{\ell =1}^{n(k)} \phi_\ell(y) \, v_\ell.
\ee
Therefore $\rank(u_k)\leq 8k+5$.
\end{theorem}

Combining this result with the approximation estimate \iref{approuk} between $u$ and $u_k$
immediately yields the following corollary.

\begin{cor}
With the same notations as in Theorem \ref{thm:rankest2x2}, we have
\be
\sup_{y\in U } \,\Bignorm{u(y) - \sum_{\ell =1}^{n(k)} \phi_\ell(y) \, v_\ell }_{V} \leq  \frac{ \norm{\bar{A}^{-1}f}_V }{1 - \theta} \, \theta^{k+1}\,,
\ee
and consequently $d_n\bigl(u(U) \bigr)_V \lesssim \exp({-\frac{\abs{\ln \theta}}8 n})$.
\end{cor}

We begin by introducing the notation 
\be  
 \textstyle\Gamma_1 := [-\frac12,\frac12]\times\{0\} \,,\quad \Gamma_2 := \{0\}\times[-\frac12,\frac12] \,,
 \ee
 for the horizontal and vertical components of the skeleton $\Gamma$ of the decomposition \eqref{part2x2}.
The space $V_\Gamma$ introduced at the beginning of this section can be decomposed into the three closed subspaces
\be
\begin{gathered}  \label{hatVspaces}
 \hat{V}_1 := \{  v\in V_\Gamma \colon \text{$v|_{\Gamma_1}$ and $v|_{\Gamma_2}$ are even} \} \,, \\
  \hat{V}_2 := \{  v\in V_\Gamma \colon \text{$v|_{\Gamma_1}$ odd, $v|_{\Gamma_2}=0$} \} \,, \\
  \hat{V}_3 := \{  v\in V_\Gamma \colon \text{$v|_{\Gamma_2}$ odd, $v|_{\Gamma_1}=0$} \}  \,.
\end{gathered}
\ee

The following result decribes the action of
$E$ on these spaces. Its proof is immediate from symmetry arguments 
and is therefore omitted.

\begin{prop}\label{prop:symm}
For the harmonic extension $Ew$ of $w\in V_\Gamma$, we have the following symmetry and antisymmetry properties depending on the particular subspaces: if $w \in \hat{V}_1$, then
$Ew(-x_1,x_2) = E w(x_1,x_2)$ and $Ew(x_1,-x_2) = E w(x_1,x_2)$ for a.e.\ $x \in D$; if $w \in\hat{V}_2$, then $Ew(-x_1,x_2) = -Ew(x_1,x_2)$ and $Ew(x_1,-x_2) = Ew(x_1,x_2)$; 
and if $w \in\hat{V}_3$, then $Ew(x_1, -x_2) = -Ew(x_1,x_2)$ and $Ew(-x_1,x_2) = Ew(x_1,x_2)$.
\end{prop}

For the proofs of the following results, it will be convenient to introduce the notation
\be
  (u, v)_{\Gamma, i} := \int_{D_i} \nabla Eu \cdot \nabla E v\,dx  \,,
\ee
so that $ (u, v)_{\Gamma} =\sum_{i=1}^4  (u, v)_{\Gamma, i}$.

\begin{prop}\label{prop:orth}
The spaces $\hat{V}_i$, $i=1,2,3$, form an orthogonal decomposition of $V_\Gamma$ with respect to the inner product $( \cdot,\cdot)_\Gamma$.
\end{prop}

\begin{proof}
Note first that a unique decomposition into even and odd parts of this form always exists.
To show orthogonality, for symmetry reasons, it suffices to consider the pairings $\hat{V}_1$, $\hat{V}_2$ and $\hat{V}_2$, $\hat{V}_3$. If either $u\in\hat{V}_3$, $v\in\hat{V}_2$ or $u\in\hat{V}_1$, $v\in\hat{V}_2$, Proposition \ref{prop:symm} gives 
\be
  (u, v)_{\Gamma,2} = - (u, v)_{\Gamma, 1} \,,\quad (u,v)_{\Gamma,3} = - (u,v)_{\Gamma, 4} \,,
\ee
and consequently $(u,v)_\Gamma = 0$.
\end{proof}

We now consider the following operators:
\begin{align*}
  H_0 &:= \bar{S},   &  H_2 &:= S_1 + S_2 - S_3 - S_4,   \\
  H_1 &:= S_1 - S_2 - S_3 + S_4,  &   H_3 &:= S_1 - S_2 + S_3 - S_4  \,. 
\end{align*}
For convenience, we set $G_i := \theta \bar{S}^{-1} H_i$, $i=0,\ldots,3$.
For each $k$, the partial sums $u_{k,\Gamma}$ introduced in \eqref{neumannpartialif} can be expressed in terms of the operators $G_i$.
We introduce the new variables 
\begin{align*}
  z_0=z_0(y) & :=\frac14(y_1+y_2+y_3+y_4),  &   z_2=z_2(y) &:= \frac14(y_1+y_2 - y_3-y_4) \,,  \\[6pt]
  z_1=z_1(y) &:= \frac14(y_1-y_2-y_3+y_4) , &  z_3 =z_3(y)&:=\frac14(y_1-y_2+y_3-y_4)  \,,
\end{align*}
so that
\be
\sum_{i=0}^{3} z_i G_i=\sum_{i=1}^4 y_i(\theta \bar{S}^{-1}S_i).
\ee
In view of \iref{neumannpartialif}, we thus obtain the representation
\be \label{interfacez}
  u_{k,\Gamma}(z) = \sum_{\ell=0}^k \Bigl( \sum_{i=0}^{3} z_i G_i\Bigr)^\ell  g_\Gamma\,.  
\ee
As we shall now show, the operators $G_i$ act on the spaces \eqref{hatVspaces} in a very particular way, which will eventually allow us to obtain the desired rank estimate.

\begin{lemma}\label{lmm:zero}
One has $G_i(\hat{V}_i) = \{0\}$ for $i=1,2,3$.
\end{lemma}

\begin{proof}
Since $\bar{S}$ is an isomorphism from $V_\Gamma$ to $V'_\Gamma$, we 
prove the equivalent statement that  $H_i(\hat{V}_i)=\{0\}$ for $i=1,2,3$.
Let $u_1 \in \hat{V}_1$ and $v \in V_\Gamma$. Then by definition,
\be  \label{eq:zero_even}
   \langle H_1 u_1, v\rangle =  (u_1, v)_{\Gamma,1 } - (u_1, v)_{\Gamma,2} - (u_1,v)_{\Gamma,3} + (u_1, v)_{\Gamma,4}.
\ee
If $v\in\hat{V}_1$, the four integrals on the right hand side all give the same value, and hence $\langle H_1 u_1, v\rangle =0$; if $v\in\hat{V}_2$ or $v\in \hat{V}_3$, then using a change of variables and the above antisymmetry properties, we obtain
\be
   (u_1, v)_{\Gamma, 4} = - (u_1, v)_{\Gamma, 1} \,, \quad  (u_1, v)_{\Gamma, 3} = - (u_1, v)_{\Gamma, 2} \,.
\ee
We have thus shown $\langle H_1 u_1, v\rangle = 0$ for any $v\in V_\Gamma$, or equivalently 
that $H_1(\hat V_1)=\{0\}$. Next, let $u_2\in\hat{V}_2$. Then 
\be
   \langle H_2 u_2, v\rangle = (u_2, v)_{\Gamma, 1} + (u_2, v)_{\Gamma, 2} - (u_2, v)_{\Gamma,3} - (u_2, v)_{\Gamma, 4} \,. 
\ee
If $v\in \hat{V}_1$ or $v\in\hat{V}_3$, we have
\be
   (u_2, v)_{\Gamma, 2} = - (u_2, v)_{\Gamma,1} \,, \quad (u_2, v)_{\Gamma, 4} = - (u_2, v)_{\Gamma, 3} \,. 
\ee
If $v\in\hat{V}_2$, again all four integrals give the same value. Thus for all $v\in V_\Gamma$, $\langle H_2 u_2,v\rangle = 0$,
 or equivalently 
$H_2(\hat V_2)=\{0\}$. 
The statement $H_3(\hat{V}_3)=\{0\}$ follows by exchanging the role of the coordinates $(x_1,x_2)$.
\end{proof}

\begin{lemma}\label{lmm:one}
One has
\be  G_2( \hat{V}_1), G_1 ( \hat{V}_2) \subset \hat{V}_3\,,\quad G_3(\hat{V}_1 ), G_1 (\hat{V}_3 )\subset \hat{V}_2\,, \quad\text{and}\quad G_3 (\hat{V}_2 ), G_2 (\hat{V}_3) \subset \hat{V}_1\,. \ee
\end{lemma}

\begin{proof}
By the symmetries in the problem, it suffices to consider $G_2(\hat{V}_1)$, $G_1(\hat{V}_2)$, and $G_3(\hat{V}_2)$.
We begin with the following observation: if $u_j,v_j \in \hat{V}_j$, then 
\begin{equation}\label{eq:obs}   \langle H_i u_j,v_j\rangle = 0\,,\quad \text{for all $i,j\in\{1,2,3\}$.}  
\end{equation}
To see this, we argue as in the proof of Lemma \ref{lmm:zero}, noting that
\be     (u_j, v_j)_{\Gamma, n} = (u_j, v_j)_{\Gamma, m}  \,,\quad n, m\in\{1,\ldots,4\}\,, \ee
since $E u_j$ and $E v_j$ have the same symmetry or antisymmetry properties. We now treat the three above cases.
\medskip

\noindent(i) $G_2(\hat{V}_1)\subset \hat V_3$: let $v_1\in \hat{V}_1$. 
In view of Proposition \ref{prop:orth}, we need to show that $H_2 v_1$ annihilates the elements in $\hat V_1$ and $\hat V_2$.
With $w_1\in \hat V_1$ and $w_2\in \hat V_2$,  we have $\langle H_2 v_1, w_2\rangle = \langle v, H_2 w_2\rangle = 0$
from Lemma \ref{lmm:zero}, and $\langle H_2 v_1, w_1\rangle=0$ follows from \eqref{eq:obs}. 
Altogether, this gives $G_2(\hat{V}_1) \subset \hat{V}_3$.
\medskip

\noindent(ii) $G_1(\hat{V}_2)\subset \hat V_3$: let $v_2\in\hat{V}_2$. With $w_1\in \hat V_1$ and $w_2\in \hat V_2$, we again have $\langle H_1 v_2, w_1\rangle = \langle v_2, H_1 w_1\rangle = 0$, and $\langle H_1 v_2, w_2\rangle = 0$ by \eqref{eq:obs}, showing $G_1(\hat{V}_2) \subset \hat{V}_3$. 
\medskip

\noindent(iii) $G_3(\hat{V}_2)\subset \hat V_1$: let $v_2\in\hat{V}_2$. With $w_2\in \hat V_2$ and $w_3\in \hat V_3$, we have $\langle H_3 v_2, w_2\rangle = 0$ by \eqref{eq:obs} and $\langle H_3 v_2, w_3\rangle = \langle v_2, H_3 w_3\rangle = 0$, which shows $G_3(\hat{V}_2)\subset \hat{V}_1$.
\end{proof}

\begin{lemma}\label{lmm:two}
For any $v_2\in \hat{V}_2$, $v_3\in\hat{V}_3$, one has
\be  G_2 G_3 v_2 = G_1 v_2 \,, \quad G_3 G_2 v_3 = G_1 v_3 \,,\quad  G_3^2 v_2 = v_2 \,,\quad G_2^2 v_3 = v_3 \,.    \ee
\end{lemma}

\begin{proof}
By symmetry arguments, it suffices to show the first and the third statement. Let $p_3:=G_2 G_3 v_2\in \hat V_3$, $p_1:=G_3 v_2\in \hat V_1$, and $q_3 := G_1 v_2\in \hat V_3$. The first statement now reads $p_3=q_3$.
By definition,
\be 
\label{eq:varform1}  \langle \bar{S} p_3, u\rangle = \langle H_2 p_1,u\rangle, \quad u\in V_\Gamma,
\ee
\be
\langle \bar{S} p_1,v\rangle = \langle H_3 v_2, v\rangle\,,
\quad v\in V_\Gamma \,, 
\label{eq:varform11}
\ee
as well as
\be\label{eq:varform2}
  \langle \bar{S} q_3, w\rangle = \langle H_1 v_2 ,w\rangle \,,\quad w\in V_\Gamma \,.
\ee
For given $u \in \hat V_3$, we now set 
\be
  v:=(\Chi_{\{x\in D\colon x_2\leq 0\}} - \Chi_{\{x\in D\colon x_2> 0\}})u \in \hat{V}_{1} \,.
\ee
For this choice of $v$, the definition of $\{H_1,H_2,H_3\}$ shows that we have $\langle H_2 p_1, u\rangle = \langle \bar{S}p_1,v\rangle$ as well as $\langle H_3 v_2, v\rangle = \langle H_1v_2, u\rangle$. Using \eqref{eq:varform1} and \eqref{eq:varform11}
this implies $\langle \bar{S} p_3, u\rangle = \langle H_1 v_2, u\rangle$ for any $u\in\hat{V}_3$. By \eqref{eq:varform2}, we thus reach
\be
\<\bar S p_3,u\>=\<\bar S q_3,u\>, \quad u\in \hat V_3.
\ee
By Lemma \ref{lmm:zero}, both sides in the above equality vanish in the the case where $u\in \hat V_1\oplus \hat V_2$.
Since $\bar S$ is an isomorphism, it follows that $p_3 = q_3$.

Let $p_2:= G^2_3 v_2$ and $q_1 := G_3 v_2$. The third statement reads $p_2 = v_2$. By definition of $p_2$ and $q_1$,
\be   \label{eq:varform3}
\langle \bar{S} p_2, u\rangle = \langle H_3 q_1, u\rangle,  \quad u\in V_\Gamma\,,
\ee
\be   
\label{eq:varform4}
\langle \bar{S} q_1,v\rangle = \langle H_3 v_2, v\rangle\,,\quad v\in V_\Gamma\,.
\ee
For $u\in\hat{V}_2$, we set 
\be
v:=(\Chi_{\{x\in D\colon x_1\leq 0\}} - \Chi_{\{x\in D\colon x_1> 0\}})u \in \hat{V}_{1}.
\ee 
For this choice of $v$, the definition of $H_3$ shows that we have $\langle H_3 q_1, u\rangle=\langle \bar{S} q_1,v\rangle$ and $\langle H_3 v_2, v\rangle = \langle\bar{S}v_2,u\rangle$. Using \eqref{eq:varform3} and  \eqref{eq:varform4}, we reach
\be
\langle \bar{S} p_2,u\rangle = \langle \bar{S} v_2,u\rangle, \quad u\in\hat{V}_2.
\ee
By Lemma \ref{lmm:zero}, both sides in the above equality vanish in the the case where $u\in \hat V_1\oplus \hat V_3$.
Since $\bar S$ is an isomorphism, it follows that $p_2 = v_2$.
\end{proof}

Recall that our aim is to bound the rank of the representation \eqref{interfacez} for each order $k$, which is bounded by the number of linearly independent elements of $V_\Gamma$ that are created applying compositions of the operators $G_i$, $i=1,2,3$, up to order $k$, to $g_\Gamma$.
In other words, defining
\be  F_k(v) :=  \{ v \} \cup \bigl\{  G_{i_1}\cdots G_{i_j}  v\colon 1 \leq j\leq k,\; i_1,\ldots,i_j \in \{ 1,2,3 \}  \bigr\}  \ee
for any given $v\in V_\Gamma$,
we have $\rank(u_{k,\Gamma}) \leq n(k) := \dim \linspan F_k(g_\Gamma)$. 
Note that this strategy for obtaining rank bounds can be regarded as the opposite of the one used in Section \ref{sec:powerexp}, which was based on instead grouping terms in the partial sums $u_k$ according to common polynomial degrees.

One has the trivial bound 
\be
n(k) \leq \# F_k(g_\Gamma) \leq 3^k + 1.
\ee 
Using the decomposition 
\be
g_\Gamma=g_1+g_2+g_3,
\ee
with $g_i\in \hat V_i$, we obviously have
\be
 \linspan F_k(g_\Gamma) \subset  \linspan\{ F_k(g_1)\cup F_k(g_2)\cup F_k(g_3)\},
\ee
and therefore 
\be
n(k) \leq n_1(k)+n_2(k)+n_3(k), \quad n_i(k):= \dim \linspan F_k(g_i) .
\ee
Lemmas \ref{lmm:zero} and \ref{lmm:one} allow us to bound the number
of non-zero elements in $F_k(g_i)$, and therefore $n_i(k)$, by $2^{k+1}-1$. We thus obtain the
slightly better estimate
\be 
n(k) \leq 3( 2^{k+1} -1).
\ee 
Note that these estimates are both much weaker than the bound of order $k^3$ that we have for $\rank(u_k)$ as a consequence of the results in Section \ref{sec:powerexp}. However, with the help of Lemma \ref{lmm:two}, 
we now show that the number of non-zero elements in each $F_k(g_i)$
in fact grows only \emph{linearly} in $k$, 
which explains our initial numerical observations and leads us to the proof of Theorem \ref{thm:rankest2x2}.

\begin{lemma}\label{lmm:graph}
Let $g_2\in\hat{V_2}$. Then for each $k \in \N$, the only non-zero elements in $F_k(g_2)$
are 
\begin{multline} \label{graphsets}
  \bigl\{  G_1^i g_2 \colon 0\leq i \leq k \bigr\}  
      \cup \bigl\{  G_2 G_1^{i} g_2 \colon \text{$i$ {\rm odd,} $1\leq i < k$} \bigr\}
      \cup \bigl\{  G_3 G_1^{i} g_2 \colon \text{$i$ {\rm even,} $0\leq i < k$} \bigr\}   \,. 
\end{multline}
Let $g_3\in\hat{V_3}$. Then for each $k \in \N$, the only non-zero elements in $F_k(g_3)$
are 
\begin{multline} \label{graphsetsex}
  \bigl\{  G_1^i g_3 \colon 0\leq i \leq k \bigr\}  
      \cup \bigl\{  G_3 G_1^{i} g_3 \colon \text{$i$ {\rm odd,} $1\leq i < k$} \bigr\}
      \cup \bigl\{  G_2 G_1^{i} g_3 \colon \text{$i$ {\rm even,} $0\leq i < k$} \bigr\} .
\end{multline}
\end{lemma}

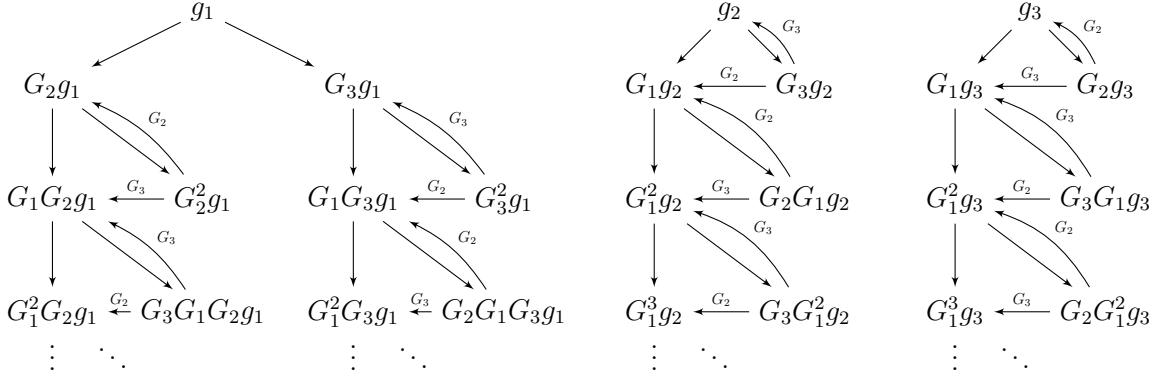
\begin{figure}
\centering
\begin{tikzpicture}[every node/.style={scale=.9}, scale=1]
\tikzset{arr/.style = {->,> = latex'}}
\node (g1) at  (-1,0) {$g_1$};
\node (g2) at  (6,0) {$g_2$};
\node (g3) at  (10,0) {$g_3$};

\node (G2g1) at (-3,-1) {$G_2 g_1$};
\node (G3g1) at (1,-1) {$G_3 g_1$};

\node (G1G2g1) at (-3,-2.5) {$G_1G_2g_1$};
\node (GG2g1) at (-1,-2.5) {$G_2^2g_1$};
\node (GG1G2g1) at (-3,-4) {$G_1^2 G_2g_1$};
\node (G3G1G2g1) at (-1,-4) {$G_3 G_1G_2g_1$};
\node (dots111) at (-3,-4.5) {$\vdots$};
\node (dots1112) at (-2.2,-4.5) {$\ddots$};

\node (G1G3g1) at (1,-2.5) {$G_1G_3g_1$};
\node (GG3g1) at (3,-2.5) {$G_3^2g_1$};
\node (GG1G3g1) at (1,-4) {$G_1^2 G_3g_1$};
\node (G2G1G3g1) at (3,-4) {$G_2 G_1G_3g_1$};
\node (dots111) at (1,-4.5) {$\vdots$};
\node (dots1113) at (1.8,-4.5) {$\ddots$};

\node (G1g2) at (5,-1) {$G_1 g_2$};
\node (G3g2) at (7,-1) {$G_3 g_2$};
\node (GG1g2) at (5,-2.5) {$G_1^2 g_2$};
\node (G2G1g2) at (7,-2.5) {$G_2 G_1g_2$};
\node (GGG1g2) at (5,-4) {$G_1^3 g_2$};
\node (dots2111) at (5,-4.5) {$\vdots$};
\node (G3GG1g2) at (7,-4) {$G_3G_1^2g_2$};
\node (dots21113) at (5.8,-4.5) {$\ddots$};

\node (G1g3) at (9,-1) {$G_1g_3$};
\node (G2g3) at (11,-1) {$G_2g_3$};
\node (GG1g3) at (9,-2.5) {$G_1^2 g_3$};
\node (G3G1g3) at (11,-2.5) {$G_3 G_1g_3$};
\node (GGG1g3) at (9,-4) {$G_1^3 g_3$};
\node (dots3111) at (9,-4.5) {$\vdots$};
\node (G2GG1g3) at (11,-4) {$G_2G_1^2g_3$};
\node (dots31112) at (9.8,-4.5) {$\ddots$};
\begin{scope}[every node/.style={scale=.5}]
\draw[arr] (g1) edge (G2g1);

\draw[arr] (G2g1) edge (G1G2g1);
\draw[arr] (G2g1) edge (GG2g1);

\draw[arr] (GG2g1)  edge[bend right=15] node[above right] {${G_2}$} (G2g1);
\draw[arr] (GG2g1) edge node[above] {$G_3$} (G1G2g1);

\draw[arr] (G1G2g1) edge (GG1G2g1);
\draw[arr] (G1G2g1) edge (G3G1G2g1);

\draw[arr] (G3G1G2g1)  edge[bend right=15] node[above right] {${G_3}$} (G1G2g1);
\draw[arr] (G3G1G2g1) edge node[above] {$G_2$} (GG1G2g1);

\draw[arr] (g1) edge (G3g1);

\draw[arr] (G3g1) edge (G1G3g1);
\draw[arr] (G3g1) edge (GG3g1);

\draw[arr] (GG3g1)  edge[bend right=15] node[above right] {${G_3}$} (G3g1);
\draw[arr] (GG3g1) edge node[above] {$G_2$} (G1G3g1);

\draw[arr] (G1G3g1) edge (GG1G3g1);
\draw[arr] (G1G3g1) edge (G2G1G3g1);

\draw[arr] (G2G1G3g1)  edge[bend right=15] node[above right] {${G_2}$} (G1G3g1);
\draw[arr] (G2G1G3g1) edge node[above] {$G_3$} (GG1G3g1);

\draw[arr] (g2) edge (G1g2);
\draw[arr] (g2) edge (G3g2);

\draw[arr] (G3g2) edge[bend right=20] node[above right] {${G_3}$} (g2);
\draw[arr] (G3g2) edge node[above] {$G_2$} (G1g2);

\draw[arr] (G1g2) edge (GG1g2);
\draw[arr] (G1g2) edge (G2G1g2);

\draw[arr] (G2G1g2)  edge[bend right=20] node[above right] {${G_2}$} (G1g2);
\draw[arr] (G2G1g2) edge node[above] {$G_3$} (GG1g2);

\draw[arr] (GG1g2) edge (GGG1g2);
\draw[arr] (GG1g2) edge[scale=0.9] (G3GG1g2);

\draw[arr] (G3GG1g2)  edge[bend right=20] node[above right] {${G_3}$} (GG1g2);
\draw[arr] (G3GG1g2) edge node[above] {$G_2$} (GGG1g2);

\draw[arr] (g3) edge (G1g3);
\draw[arr] (g3) edge (G2g3);

\draw[arr] (G2g3)  edge[bend right=20] node[above right] {${G_2}$} (g3);
\draw[arr] (G2g3) edge node[above] {$G_3$} (G1g3);

\draw[arr] (G1g3) edge (GG1g3);
\draw[arr] (G1g3) edge (G3G1g3);

\draw[arr] (G3G1g3)  edge[bend right=20] node[above right] {${G_3}$} (G1g3);
\draw[arr] (G3G1g3) edge node[above] {$G_2$} (GG1g3);

\draw[arr] (GG1g3) edge (GGG1g3);
\draw[arr] (GG1g3) edge[scale=0.9] (G2GG1g3);

\draw[arr] (G2GG1g3)  edge[bend right=20] node[above right] {${G_2}$} (GG1g3);
\draw[arr] (G2GG1g3) edge node[above] {$G_3$} (GGG1g3);
\end{scope}
\end{tikzpicture}
\caption{Illustration of the proof of Lemma \ref{lmm:graph}, where $g_\Gamma = g_1 + g_2 + g_3$.}\label{proofgraph}
\end{figure}
The idea behind the proof of Lemma \ref{lmm:graph}, which combines Lemmas \ref{lmm:zero}, \ref{lmm:one}, and \ref{lmm:two}, is illustrated on the two
graphs on the right of Figure \ref{proofgraph}. The nodes in these graphs represent the only possible non-zero functions that can be created by repeated application of the operators $G_i$, $i=1,2,3$, to a given $g_2\in \hat V_2$ or $g_3\in \hat V_3$. The first graph on the left 
represents the only possible non-zero functions that can be created by repeated application of the $G_i$
to a given $g_1\in \hat V_1$. This first graph is actually obtained by applying the third and second graphs
on its first two nodes $G_2g_1\in \hat V_3$ and $G_3g_1\in \hat V_2$.

\begin{proof} It suffices to prove the result for $F_k(g_2)$, since the result for $F_k(g_3)$ follows by symmetry.
For each $k$, we use the abbreviation $F_k := F_k(g_2)$ and denote the set in \eqref{graphsets} by 
\be 
 R_k = R_{k,1}\cup R_{k,2}\cup R_{k,3} \,.
\ee
We show that $F_{k}= R_{k}\cup \{0\}$ for all $k\geq 1$. In the case $k=1$,  we have
\be
F_k(g_2)=\{g_2,  G_1 g_2, G_2 g_2, G_3 g_2 \} = \{ g_2, G_1 g_2, 0, G_3 g_2 \},
\ee
which proves the claim in this case. We assume the claim holds for a $k \geq 1$, and we thus need to show
that $F_{k+1}=R_{k+1}\cup \{0\}$. The inclusion $R_{k+1}\cup \{0\}\subset F_{k+1}$ follows trivially from
the definitions. To show that $F_{k+1} \subset R_{k+1}\cup \{0\}$, note first that from the definition of $F_k$, we have 
\be
F_{k+1} =  \{ g_2 \} \cup \bigcup_{i=1}^3 G_i (F_k) 
\ee
and hence, by our induction hypothesis,
\be   F_{k+1} =\{ g_2 \} \cup \bigcup_{i=1}^3 G_i (R_k \cup \{0\} ) \,. \ee
By Lemma \ref{lmm:one},
$G^i_1 g_2\in V_2$ for $i$ even and $G^i_1 g_2\in V_3$ for $i$ odd. Thus 
\begin{gather*} 
G_1 (R_{k,1}) \subset R_{k,1} \cup \{ G^{k+1}_1 g_2\} = R_{k+1,1}\,,\\
G_2( R_{k,1}) = \{ 0 \}\cup \{ G_2 G^i_1 g_2\colon \text{$i$ odd, $1 \leq i \leq k$} \} = \{ 0 \}\cup R_{k+1,2} \,,\\
G_3 (R_{k,1}) = \{ 0 \}\cup \{ G_3 G^i_1 g_2\colon \text{$i$ even, $0\leq i\leq k$} \} = \{ 0 \}\cup R_{k+1,3} \,.
\end{gather*}
By Lemma \ref{lmm:one}, $G_1 (R_{k,2}) = G_1 (R_{k,3}) = \{ 0 \}$ and by Lemma \ref{lmm:two},
$G_i (R_{k,j}) \subset \{0\}\cup R_{k+1,1}$ for $i,j\in\{2,3\}$. Since furthermore $\{g_2\} \subset R_{k+1,1}$ for each $k$, we arrive at $F_{k+1}
=R_{k+1}\cup\{0\}$.
\end{proof}

With Lemma \ref{lmm:graph} at hand, we can now prove Theorem \ref{thm:rankest2x2}.

\begin{proof}[Proof of Theorem \ref{thm:rankest2x2}]
Invoking \iref{rankif}, we know that 
\be
\rank(u_k)\leq 4+\rank (u_{k,\Gamma}),
\ee
and by the representation \eqref{interfacez} the rank of $u_{k,\Gamma}$ is bounded by
the number of linearly independent vectors in $F_k(g_\Gamma)$.

After an orthogonal decomposition of $g_\Gamma$ in the form $g_\Gamma = g_1 + g_2 + g_3$ with $g_i\in\hat{V}_i$, Lemma \ref{lmm:graph} can be applied separately
to $G_3g_1$, $G_2g_1$, $g_2$, $g_3$.  This shows that the number of non-zero elements in $F_k(g_i)$ is at most 
$(2k+1)$ for $i=2$ and $i=3$, and at most $1+ 2(2k-1)$ for $i=1$. In conclusion,
$F_k(g_\Gamma)$ contains at most $1 + 2(2k-1) + 2(2k+1) = 8k + 1$ linearly independent elements, which confirms
the estimate $8k+5$ for $\rank(u_k)$.
\end{proof}

\section{Numerical construction of low-rank approximations}\label{sec:numerics}

The numerical scheme used for the computational examples in this work is based on the iteration
\be\label{fpiter}
   u_0 := \bar{A}^{-1} f = g \,,\qquad  u_{k} := \bar{A}^{-1} \Bigl(  f + \sum_{i=1}^d y_i A_i u_{k-1} \Bigr)  = g + \sum_{i=1}^d  y_i B_i u_{k-1}\,, \; k \in \N \,.
\ee
As a consequence of Proposition \ref{propcontraction}, this is a fixed point iteration with linear convergence in $L^\infty(U,V)$. The iterates $u_k$ are precisely the partial sums \eqref{neumannpartial}, that is,
\be
u_k(y) = \sum_{\abs{\nu}\leq k} t_\nu y^\nu = \sum_{\ell = 0}^k  \Bigl(  \sum_{i=1}^d y_i B_i \Bigr)^\ell g \,.
\ee
Since we shall use the singular value decomposition to obtain low-rank approximations, we regard $u$ and the $u_k$ as elements of $L^2(U,V) \simeq V \otimes L^2(U)$ with the uniform probability measure $\mu$ on $U$, and consider convergence in this norm. 

We first introduce a discretization in the parametric variable $y$, based on a choice of basis for $L^2(U)$. We use
here the orthonormal tensor product Legendre polynomials $L_\nu$, $\nu \in \N_0^d$. The set of basis indices used in the computation is of the form 
\be 
 \Lambda_J := \{  \nu \in \N_0^d \colon \abs{\nu}\leq J\} \,, 
\ee
so that $\linspan\{ L_\nu \}_{\nu\in \Lambda_J}=\P_J$, the space
of polynomials of total degree at most $J$. For each $i$, the operator corresponding to multiplication by $y_i$ on $L^2(U)$ is replaced by its Galerkin discretization $\mathbf{M}_i := \bigl(  \langle y_i L_\nu, L_\mu \rangle_{L^2(U)} \bigr)_{\mu,\nu \in \Lambda_J }$, which is a bidiagonal matrix
(we refer to \cite{SG} for further details). Note that in this semidiscrete setting, as long as $k < J$, all operations in \eqref{fpiter} are represented exactly, and consequently we \emph{exactly recover} $u_k$ for $k < J$. As $k$ increases further, $u_k$ converges to the Galerkin projection $G_Ju$ of $u$ onto
the space  $V\otimes \P_J$. This projection  is defined by
\be
\int_U \int_D a(x,y) \nabla_x G_Ju(x,y)\nabla_x v(x,y) \,dx\,d\mu(y)=\int_U\int_D f(x)v(x,y)\,dx\,d\mu(y), \;\, v\in V\otimes \P_J.
\ee 

The discretization is completed by replacing $V$ by a fixed finite element subspace $V_h$ of dimension $M_h$. The corresponding discretizations of $\bar{A}$, $A_i$, and $f$ are denoted by $\mathbf{\bar A}$, $\mathbf{A}_i$ and $\mathbf{f}$, respectively. The iterates in the discretised version of \eqref{fpiter} can then be regarded as matrices $\mathbf{u}_k \in \R^{M_h \times N_J}$ with $N_J := \# \Lambda_J = n(d,J)$. 
As $k\to +\infty$, these iterates converge to the representation coefficients of the Galerkin projection
of $u$ on the space $V_h\otimes \P_J$.

In order to exploit the low-rank approximability, we modify the discretized version of \eqref{fpiter}
by introducing additional truncations based on the sizes of the singular values. This requires that we 
work using low-rank representations of the iterates. If $\mathbf{u}_k$ is given in low-rank form $\mathbf{u}_k = \mathbf{V}_k\boldsymbol{\Phi}_k^T$ with $\mathbf{V}_k \in \R^{M_h\times r_k}$, $\boldsymbol{\Phi}_k \in \R^{N_J\times r_k}$ with $r_k\leq \min\{M_h,N_J\}$, one step of the discretized iteration can be done by computing
\be\label{discrfpiter}
   \mathbf{\tilde u}_{k+1} := (\mathbf{\bar A}^{-1} \mathbf{f} ) \mathbf{e}_0^T + \sum_{i=1}^d (\mathbf{\bar A}^{-1} \mathbf{A}_i \mathbf{V}_k ) ( \mathbf{M}_i \boldsymbol{\Phi}_k)^T \,.
\ee
Here $\mathbf{e}_0 := (\delta_{0 \nu})_{\nu\in\Lambda_J}$, regarded as a column vector. Thus \eqref{discrfpiter} yields a representation of the form $\mathbf{\tilde u}_{k+1} = \mathbf{\tilde V}_{k+1} \boldsymbol{\tilde \Phi}_{k+1}^T$ with formal \emph{representation rank} $\tilde r_{k+1} = n r_k + 1$, which may be larger than $\rank(\mathbf{\tilde u}_{k+1})$. 
To detect and computationally exploit further low-rank structure in $\mathbf{\tilde u}_{k+1}$, we truncate its singular value decomposition (SVD) up to a prescribed tolerance $\varepsilon_k$ in the Frobenius norm. Note that in doing so, we need to use the norm induced by $\bar{A}$ on $V_h$. To achieve this, we compute a sparse Cholesky decomposition $\mathbf{\bar A} = \mathbf{L}\mathbf{L}^T$, perform the truncated SVD on $\mathbf{L}^{T} \mathbf{\tilde u}_{k+1}$, and pre-multiply the result by $\mathbf{L}^{-T}$. The actual SVD can performed at a total cost of order $\max\{ M_h,N_J\}\,  \tilde r_{k+1}^2$ by orthogonalizing the columns of $\mathbf{L}^T \mathbf{\tilde V}_{k+1}$ and $\boldsymbol{\tilde\Phi}_{k+1}$ and then performing an SVD of a matrix of size $\tilde r_{k+1}\times \tilde r_{k+1}$. We define $\mathbf{u}_{k+1} = \mathbf{V}_{k+1} \boldsymbol{\Phi}_{k+1}^T$, with rank $r_{k+1}\leq \tilde r_{k+1}$, as the result of this rank truncation. 

Note that these new iterates $\mathbf{u}_{k}$ differ from the initial iterates defined by only performing 
the discretised version of \eqref{fpiter}. With appropriately chosen truncation parameters $\varepsilon_k$, 
as $k\to \infty$ the iterates $\mathbf{u}_k$ converge linearly to the representation coefficients of the Galerkin projection
of $u$ on $V_h\otimes \P_J$, similar to the initial iterates.
The precise choice of the $\varepsilon_k$ is crucial not only for ensuring convergence, but also for the efficiency of the resulting scheme. A choice that can be shown to achieve this balance is given in \cite{BD}. In our present setting, our aim is less to obtain an efficient method than to closely reproduce the exact $u_k$, and we thus use $\varepsilon_k = 10^{-15}$, i.e., a truncation tolerance close to machine precision, for all tests in this work.

Concerning the choice of $V_h$, no triangles in the mesh should be intersected by subdomain boundaries, since this has a negative effect on the approximation error. Note further that if this is taken into account, as a consequence of \eqref{rankif}, the maximum rank of the $\mathbf{u}_k$ that can be observed numerically is bounded by $d$ plus the number of degrees of freedom on the skeleton $\Gamma$. This means in particular that in order to obtain results that reflect the actual decay of the singular values of $u$, when using only uniformly refined triangulations one needs very fine meshes to produce sufficiently many degrees of freedom on $\Gamma$. For this reason, in our numerical experiments we use meshes with a strong gradation towards $\Gamma$.

\section{Conclusions and outlook}

We have constructed low-rank approximations of solutions of certain parametric 
elliptic problems, based on assumptions on additional structures in these problems. 
Our analysis reveals particular mechanisms which, for specific problems, 
may ensure significant improvements between low-rank approximations based on 
truncated polynomial expansions and optimal low-rank approximations. 

 Whereas the rather general assumption \eqref{sumAi} yields only a rather modest 
 improvement in the $n$-widths as compared to estimates obtained by a direct polynomial 
 expansion, we arrive at a much stronger estimate in the particular setting considered 
 in Section \ref{sec:2x2}. As demonstrated there, this improvement depends rather strongly on the symmetries in the geometry of the problem. 
Numerical experiments indicate that the exponential decay of the singular values is maintained 
in the case of a similar $d=m\times m$ checkerboard structure on the unit square, however with
no general proof available at the present stage.

As an illustration we reconsider in more detail the problem with $16$ parameters of Figure \ref{fig:introdecay4x4}, with a partition of $D$ analogous to \eqref{part2x2}, but with a $4\times 4$ checkerboard pattern of squares of side length $\frac14$. 
The numerical realization was done as in the numerical experiments 
in Section 4, here using total degree 5, resulting in 20\,349 Legendre coefficients. 
The spatial discretization uses piecewise linear finite elements with 112\,961 degrees of freedom, with 12\,345 of these on subdomain interfaces.

A more detailed view of the ranks of partial sums $u_k$ and of the decay of the corresponding singular values is shown in Figure \ref{fig:decay4x4}.
The observed decay of singular values is slower than in the example of Section 4, but clearly still exponential. However, here the ranks of partial sums increase \emph{faster} than linearly. This indicates that analyzing the ranks of partial sums alone, as in the $2\times 2$  checkerboard case, 
might not be sufficient to explain the exponential decay in this case, and that a finer analysis of the singular values associated with
these partial sums may be required.

\begin{figure}
\includegraphics[width=8cm]{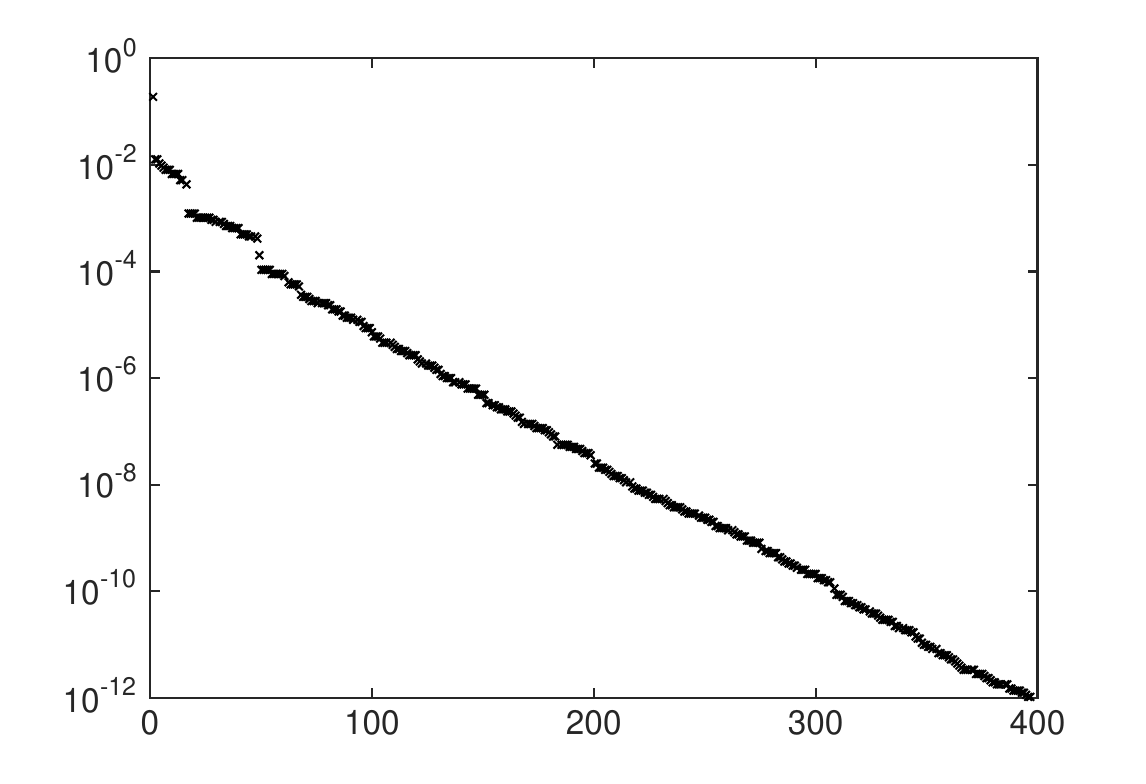}  \includegraphics[width=8cm]{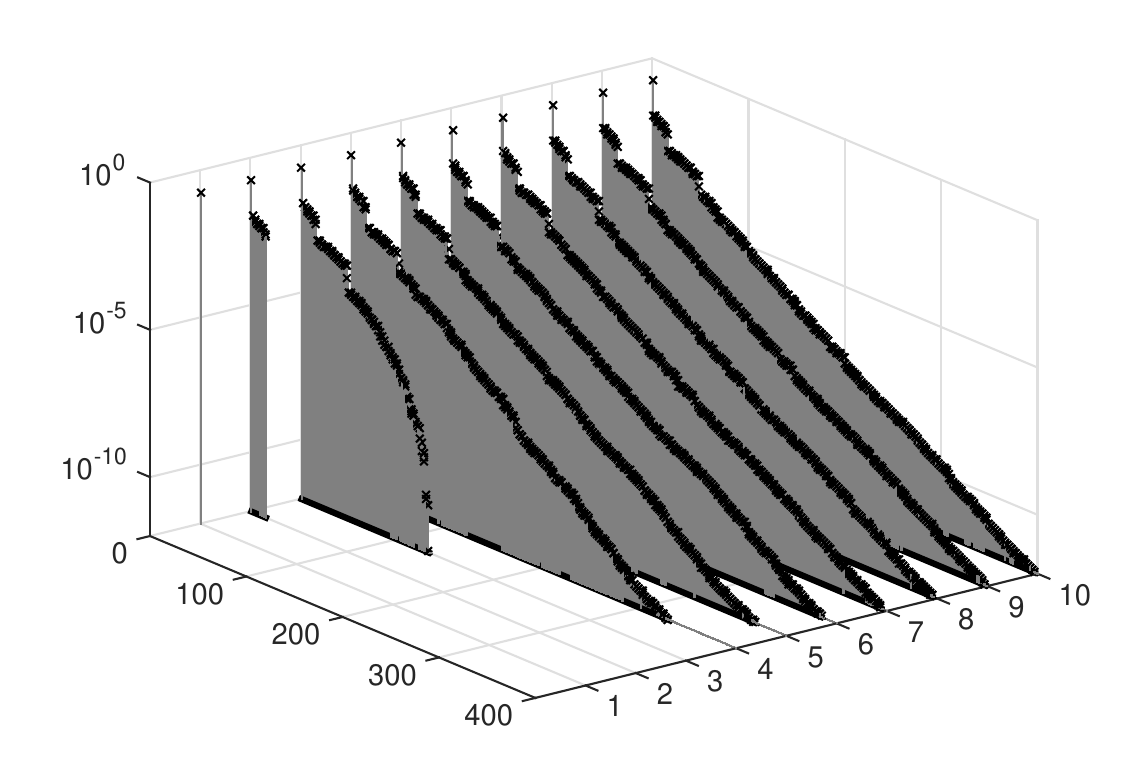} 
\caption{Decay of singular values and ranks of partial sums for 16 parameters.}
\label{fig:decay4x4}
\end{figure}

$\;$
\nl
Markus Bachmayr
\nl
UPMC Univ Paris 06, UMR 7598, Laboratoire Jacques-Louis Lions, F-75005, Paris, France
\nl
CNRS, UMR 7598, Laboratoire Jacques-Louis Lions, F-75005, Paris, France
\nl
bachmayr@ann.jussieu.fr
\nl
\nl
Albert Cohen
\nl
UPMC Univ Paris 06, UMR 7598, Laboratoire Jacques-Louis Lions, F-75005, Paris, France
\nl
CNRS, UMR 7598, Laboratoire Jacques-Louis Lions, F-75005, Paris, France
\nl
cohen@ann.jussieu.fr

\begin{thebibliography}{99}

\bibitem{ACS} R. Andreev and C. Schwab,
{\it Sparse Tensor Approximation of Parametric Eigenvalue Problems},
in Numerical Analysis of Multiscale Problems
(I.G. Graham, T.Y. Hou, O. Lakkis and R. Scheichl (Eds)),
Springer Lecture Notes in Comp. Sci. Eng., {\bf 83},
203-242, Springer Verlag, 2011.

\bibitem{BD} M. Bachmayr and W. Dahmen, 
\textit{Adaptive near-optimal rank tensor approximation for high-dimensional operator equations}, 
to appear in J.\ Foundations of Computational Mathematics, 2014, DOI 10.1007/s10208-013-9187-3.

\bibitem{BG} J. Ballani and L. Grasedyck, {\it 
Hierarchical tensor approximation of output quantities of parameter-dependent PDEs},
ANCHP Preprint, EPF Lausanne, 2014.

\bibitem{BNTT1} J. Beck, F. Nobile, L. Tamellini, and R. Tempone, {\it On the optimal polynomial approximation of
stochastic PDEs by Galerkin and collocation methods}, Mathematical Models and Methods in Applied Sciences, {\bf 22}, 1-33, 2012.

\bibitem{BNTT2} J. Beck, F. Nobile, L. Tamellini, and R.Tempone, 
{\it Convergence of quasi-optimal stochastic Galerkin methods 
for a class of PDEs with random coefficients}, Computers and Mathematics with Applications,
{\bf 67}, 732-751, 2014.

 \bibitem{BCDDPW}  P. Binev, A. Cohen, W. Dahmen, R. DeVore, G. Petrova, and P. Wojtaszczyk, {\em Convergence rates for greedy algorithms in reduced basis methods}, SIAM Journal of Mathematical Analysis, {\bf 43}, 1457-1472, 2011.

\bibitem{CCDS} A.\ Chkifa, A.\ Cohen, R.\ DeVore and C.\ Schwab,
{\it Sparse adaptive Taylor approximation algorithms for parametric and stochastic elliptic PDEs},
M2AN, {\bf 47}, 253-283, 2013.

\bibitem{CD} A.\ Cohen and R.\ DeVore, {\it Kolmogorov widths under holomorphic mappings}, to appear in
IMA Journal of Applied Mathematics, 2015.

\bibitem{CD2} A. Cohen and R. DeVore, {\it High dimensional approximation of parametric PDEs}, to appear in
Acta Numerica, 2015.

\bibitem{CDS1}  A.\ Cohen, R.\ DeVore and C.\ Schwab, 
{\it Convergence rates of best $N$-term Galerkin approximations for a class of elliptic sPDEs}, 
J.\ Foundations of Computational Mathematics, {\bf 10}, 615-646, 2010.

\bibitem{CDS2}  A.\ Cohen, R.\ DeVore and C.\ Schwab, 
{\it Analytic regularity and polynomial approximation of parametric and stochastic PDE's}, 
Analysis and Applications, {\bf 9}, 1-37, 2011.

\bibitem{DPW} R. DeVore, G. Petrova, and P. Wojtaszczyk,  {\it Greedy algorithms for reduced bases in Banach spaces}, Constructive Approximation, {\bf 37}, 455-466, 2013.

\bibitem{SG} C.\,J.\ Gittelson and C.\ Schwab, \textit{Sparse tensor discretization of high-dimensional parametric and stochastic {PDE}s}, Acta Numerica, \textbf{20}, 2011.


\bibitem{KV} M. Kahlbacher and S. Volkwein, {\it Galerkin proper orthogonal decomposition methods for parameter dependent elliptic systems}, Discussiones Mathematicae: Differential Inclusions, Control and Optimization, {\bf 27}, 95-117, 2007.

\bibitem{KS} B. Khoromskij and C. Schwab, {\it Tensor-structured Galerkin approximation of parametric and stochastic elliptic PDEs},
SIAM Journal of Scientific Computing, {\bf 33}, 364-385, 2011.

\bibitem{KT} D. Kressner and C. Tobler, {\it Low-rank tensor Krylov subspace methods 
for parameterized linear systems}, SIAM J. Matrix Anal. Appl., {\bf 32}, 1288-1316, 2011.

\bibitem{RHP} G. Rozza, D.B.P. Huynh, and A.T. Patera, {\it Reduced basis approximation and a posteriori error estimation for affinely
parametrized elliptic coercive partial differential equations},  Archives
of Computational Methods in Engineering, {\bf 15}, 229-275, 2008.

\end{thebibliography}
\end{document}